\documentclass[oneside,english,reqno]{amsart}
\usepackage{lmodern}
\usepackage[T1]{fontenc}
\usepackage[latin9]{inputenc}
\usepackage{babel}
\usepackage{amsthm}
\usepackage{amssymb}
\usepackage{graphicx}
\usepackage[unicode=true,pdfusetitle,
 bookmarks=true,bookmarksnumbered=false,bookmarksopen=false,
 breaklinks=false,pdfborder={0 0 1},backref=false,colorlinks=false]
 {hyperref}
\usepackage{breakurl}

\makeatletter
\numberwithin{equation}{section}
\numberwithin{figure}{section}
\theoremstyle{plain}
\newtheorem{thm}{\protect\theoremname}[section]
  \theoremstyle{definition}
  \newtheorem{defn}[thm]{\protect\definitionname}
  \theoremstyle{plain}
  \newtheorem{lem}[thm]{\protect\lemmaname}
  \theoremstyle{plain}
  \newtheorem{algorithm}[thm]{\protect\algorithmname}
  \theoremstyle{remark}
  \newtheorem{rem}[thm]{\protect\remarkname}
  \theoremstyle{definition}
  \newtheorem{example}[thm]{\protect\examplename}
  \theoremstyle{plain}
  \newtheorem{prop}[thm]{\protect\propositionname}

\newcommand{\intr}{\mbox{\rm int}}

\newcommand{\cone}{\mbox{\rm cone}}

\makeatother

  \providecommand{\algorithmname}{Algorithm}
  \providecommand{\definitionname}{Definition}
  \providecommand{\examplename}{Example}
  \providecommand{\lemmaname}{Lemma}
  \providecommand{\propositionname}{Proposition}
  \providecommand{\remarkname}{Remark}
\providecommand{\theoremname}{Theorem}

\begin{document}
\title[First order optimization methods with feasibility]{First order constrained optimization algorithms with feasibility updates} 

\begin{abstract}
We propose first order algorithms for convex optimization problems
where the feasible set is described by a large number of convex inequalities
that is to be explored by subgradient projections. The first algorithm
is an adaptation of a subgradient algorithm, and has convergence rate
$1/\sqrt{k}$. The second algorithm has convergence rate $1/k$ when
(1) one has linear metric inequality in the feasible set, (2) the
objective function is strongly convex, differentiable and has Lipschitz
gradient, and (3) it is easy to optimize the objective function on
the intersection of two halfspaces. This second algorithm generalizes
Haugazeau's algorithm. The third algorithm adapts the second algorithm
when condition (3) is dropped. We give examples to show that the second
algorithm performs poorly when the objective function is not strongly
convex, or when the linear metric inequality is absent.  
\end{abstract}

\author{C.H. Jeffrey Pang}

\curraddr{Department of Mathematics\\ 
National University of Singapore\\ 
Block S17 08-11\\ 
10 Lower Kent Ridge Road\\ 
Singapore 119076 }

\email{matpchj@nus.edu.sg}

\date{\today{}}


\maketitle
\tableofcontents{}

\section{Introduction}

Let $f:\mathbb{R}^{n}\to\mathbb{R}$ and $f_{j}:\mathbb{R}^{n}\to\mathbb{R}$,
where $j\in\{1,\dots,m\}$, be convex functions. Let $Q\subset\mathbb{R}^{n}$
be a closed convex set. The problem that we study in this paper is
\begin{eqnarray}
 & \min & f(x)\label{eq:main-pblm}\\
 & \mbox{s.t.} & f_{j}(x)\leq0\mbox{ for }j\in\{1,\dots,m\}\nonumber \\
 &  & x\in Q.\nonumber 
\end{eqnarray}
If $m$ is large, then it might be difficult for an algorithm to find
an $x$ satisfying the stated constraints, let alone solve the optimization
problem. We now recall material relevant with our approach for trying
to solve \eqref{eq:main-pblm}.

\subsection{Projection methods for solving feasibility problems}

For finitely many closed convex sets $C_{1},\dots,C_{m}$ in $\mathbb{R}^{n}$,
the \emph{Set Intersection Problem }(SIP) is stated as:
\begin{equation}
\mbox{(SIP):}\quad\mbox{Find }x\in C:=\bigcap_{j=1}^{m}C_{j}\mbox{, where }C\neq\emptyset.\label{eq:SIP}
\end{equation}
The SIP is also referred to as feasibility problems in the literature.
When $m$ is large, the Method of Alternating Projections (MAP) is
a reasonable way to solve the SIP. As its name suggests, the MAP finds
the sequence $\{x_{k}\}_{k=1}^{\infty}$ by projecting onto the $C_{j}$
cyclically, i.e., $x_{k+1}=P_{C_{k^{\prime}}}(x_{k})$, where $k^{\prime}$
is the number in $\{1,\dots,m\}$ such that $m$ divides $k-k^{\prime}$.
We refer the reader to \cite{BB96_survey,BR09,EsRa11}, as well as
\cite[Chapter 9]{Deustch01} and \cite[Subsubsection 4.5.4]{BZ05},
for more on the literature of using projection methods to solve the
SIP.

The convergence rate of the MAP is linear under the assumption of
linear regularity. The notion was introduced and studied by \cite{Bauschke_thesis}
(Definition 4.2.1, page 53) in a general setting of a Hilbert space.
See also \cite{BB96_survey} (Definition 5.6, page 40). Recently,
it has been studied in \cite{Deutsch_Hundal_2006a,Deutsch_Hundal_2006b,Deutsch_Hundal_2008}.
The connection with the stability under perturbation of the sets $C_{j}$
is investigated in \cite{Kruger04,Kruger_06} and other works.

Another problem closely related to the SIP is the \emph{Best Approximation
Problem} (BAP), stated as 
\begin{eqnarray}
\mbox{(BAP):}\quad & \underset{x\in X}{\min} & \frac{1}{2}\|x-x_{0}\|^{2}\label{eq:Proj-pblm}\\
 & \mbox{s.t. } & x\in C:=\bigcap_{j=1}^{m}C_{j}.\nonumber 
\end{eqnarray}
In other words, the BAP is the problem of finding the projection of
$x_{0}$ onto $C$. The BAP follows the template of \eqref{eq:main-pblm}
when $f(x)=\frac{1}{2}\|x-x_{0}\|^{2}$, $f_{j}(x)=d(x,C_{j})$ for
each $j\in\{1,\dots,m\}$, and $Q=\mathbb{R}^{n}$. Dykstra's algorithm
\cite{Dykstra83,BD86} is a projection algorithm for solving the BAP.
It was rediscovered in \cite{Han88} using mathematical programming
duality. Another algorithm is Haugazeau's algorithm \cite{Haugazeau68}
(see \cite{BauschkeCombettes11}). The convergence rate of Dykstra's
algorithm has been analyzed in the polyhedral case \cite{Deutsch_Hundal_rate_Dykstra,Xu_rate_Dykstra},
but little is known about the general convergence rates of Dykstra's
and Haugazeau's Algorithms.

For more on the background and recent developments of the MAP and
its variants, we refer the reader to \cite{BB96_survey,BR09,EsRa11},
as well as \cite[Chapter 9]{Deustch01} and \cite[Subsubsection 4.5.4]{BZ05}.

\subsection{First order algorithms and algorithms for \eqref{eq:main-pblm}}

First order methods in optimization are methods based on function
values and gradient evaluations. Even though first order methods have
a slower rate of convergence than other algorithms, the advantage
of first order algorithms is that each iteration is easy to perform.
For large scale problems, algorithms with better complexity require
too much computational effort to perform each iteration, so first
order algorithms can be the only practical method. Classical references
include \cite{Nemirovsky_Yudin,Nesterov_1983,Nesterov_1984,Nesterov_1989},
and newer references include \cite{Nesterov_book,JuditskyNemirovski_survey_a,JuditskyNemirovski_survey_b}.
See also \cite{BeckTeboulle2009}.

As far as we are aware, the problem \eqref{eq:main-pblm} where projections
are used to address the feasibility of solutions are studied in \cite{Nedic_MP_2011,Wang_Bertsekas_2015}.
In both papers, the approach is to use random projection methods,
while the second paper focuses on the generalized setting of variational
inequalities.

\subsection{Contributions of this paper}

In Section \ref{sec:subgrad-alg}, we modify the subgradient algorithm
in \cite[Section 3.2.4]{Nesterov_book} for solving \eqref{eq:main-pblm}
so that the new algorithm is more suitable for solving the problem
\eqref{eq:main-pblm} when $m$ is large. When the functions $\{f_{j}\}_{j=1}^{m}$
satisfy the linear metric inequality property in Definition \ref{def:lin-metric-ineq},
we show that projection methods can be used instead. The algorithms
in this section have $O(1/\sqrt{k})$ convergence rate to the optimal
objective value, just like the subgradient algorithm. 

The convergence of projection algorithms for the SIP \eqref{eq:SIP}
is linear when a linear metric inequality condition is satisfied.
Furthermore, the convergence of first order algorithms for strongly
convex functions with Lipschitz gradient to the objective value and
the unique optimal solution is linear. It is therefore natural to
look at the convergence rate of \eqref{eq:main-pblm} when 
\begin{enumerate}
\item [(1)] the functions $\{f_{j}\}_{j=1}^{m}$ satisfy linear metric
inequality, and 
\item [(2)] $f(\cdot)$ is strongly convex, differentiable and has Lipschitz
gradient.
\end{enumerate}
In Section \ref{sec:gen-Haugazeau}, we generalize Haugazeau's algorithm
to obtain a first order algorithm to solve \eqref{eq:main-pblm} for
the case when (1) and (2) are satisfied, and 
\begin{enumerate}
\item [(3)] $f(\cdot)$ is structured enough to optimize over the intersection
of two halfspaces. 
\end{enumerate}
Our algorithms have a $O(1/k)$ convergence rate to the optimal objective
value and $O(1/\sqrt{k})$ convergence to the optimizer. We believe
that such a convergence rate for Haugazeau's algorithm is new. 

In Section \ref{sec:gen-constrained-opt}, we propose a first order
algorithm to solve \eqref{eq:main-pblm} when (1) and (2) are satisfied,
but not (3). The convergence rate to the optimal objective value and
to the optimizer are slightly worse than the algorithms in Section
\ref{sec:gen-Haugazeau}.

In Section \ref{sec:effectiveness}, we show that in the case where
the dimension and number of constraints are large, then a $(1/k)$
convergence rate is best possible for strongly convex problems in
a model generalizing Haugazeau's algorithm, while an arbitarily slow
convergence rate applies when there is convexity but no strong convexity
in the objective function.

In Section \ref{sec:Behavior-Haugazeau}, we show that the $O(1/k)$
rate of convergence of Haugazeau's algorithm to the objective value
occurs even for a very simple example. We give a second example to
show that Haugazeau's algorithm converges arbitrarily slowly in the
absence of linear metric inequality.

\section{Preliminaries }

In this section, we recall some results that will be necessary for
the understanding of this paper. We start with strongly convex functions.
\begin{defn}
(Strongly convex functions) We say that $f:\mathbb{R}^{n}\to\mathbb{R}$
is \emph{strongly convex with convexity parameter $\mu$} if 
\[
f(y)\geq f(x)+\langle f^{\prime}(x),y-x\rangle+\frac{\mu}{2}\|x-y\|^{2}\mbox{ for all }x,y\in\mathbb{R}^{n}.
\]
Denote the set $\mathcal{S}_{\mu,L}^{1,1}$ to be the set of all functions
$f:\mathbb{R}^{n}\to\mathbb{R}$ such that $f(\cdot)$ is strongly
convex with parameter $\mu$ and $f^{\prime}(\cdot)$ is Lipschitz
with constant $L$. 
\end{defn}
We recall some standard results and notation on the method of alternating
projections that will be used in the rest of the paper. 
\begin{lem}
\label{lem:attractive-ppty-of-projection}(Attractive property of
projection) Let $C\subset\mathbb{R}^{n}$ be a closed convex set.
Then $P_{C}:X\to X$ is \emph{1-attracting with respect to $C$}:
\[
\|P_{C}(x)-x\|^{2}\leq\|x-y\|^{2}-\|P_{C}(x)-y\|^{2}\mbox{ for all }x\in\mathbb{R}^{n}\mbox{ and }y\in C.
\]
\end{lem}
\begin{defn}
(Fej\'{e}r monotone sequence) Let $C\subset\mathbb{R}^{n}$ be a
closed convex set and let $\{x_{k}\}$ be a sequence in $\mathbb{R}^{n}$.
We say that $\{x_{k}\}$ is\emph{ Fej\'{e}r monotone with respect
to $C$} if 
\[
\|x_{k+1}-c\|\leq\|x_{k}-c\|\mbox{ for all }c\in C\mbox{ and }i=1,2,\dots
\]

\end{defn}
Consider the SIP \eqref{eq:SIP} and the method of alternating projections
described shortly after. The 1-attractiveness property leads to the
Fej\'{e}r monotonicity of the sequence $\{x_{k}\}_{k=1}^{\infty}$
with respect to $C=\cap_{j=1}^{m}C_{j}$. The Fej\'{e}r monotonicity
property will be used in the proof of Theorem \ref{thm:subgrad-conv-rate}.

A stability property that guarantees the linear convergence of the
MAP is defined below.
\begin{defn}
\label{def:lin-metric-ineq}(Linear metric inequality) Let $f_{j}:\mathbb{R}^{n}\to\mathbb{R}$
be convex functions for $j\in\{1,\dots,m\}$. Let $C:=\{x:\max_{1\leq j\leq m}f_{j}(x)\leq0\}$.
Let $x\in\mathbb{R}^{n}$. If $f_{j}(x)>0$, then choose any $\bar{g}_{j}\in\partial f_{j}(x)$
and let the halfspace $H_{j}$ be 
\[
H_{j}:=\{y:f_{j}(x)+\langle\bar{g}_{j},y-x\rangle\leq0\}.
\]
Otherwise, let $H_{j}=\mathbb{R}^{n}$. We say that $\{f_{j}(\cdot)\}_{j=1}^{m}$
satisfies \emph{linear metric inequality with parameter $\kappa>0$
}if 
\begin{equation}
d(x,C)\leq\kappa\max_{1\leq j\leq m}d(x,H_{j})\mbox{ for all }x\in\mathbb{R}^{n}.\label{eq:gen-lin-metric-ineq}
\end{equation}

\end{defn}
In the case where $f_{j}(x)=d(x,C_{j})$ for some closed convex set
$C_{j}$, then $\partial f_{j}(x)=\left\{ \frac{x-P_{C_{j}}(x)}{\|x-P_{C_{j}}(x)\|}\right\} $
and $\|x-P_{C_{j}}(x)\|=d(x,C_{j})$ (This fact is well known. See
for example \cite[Proposition 18.22]{BauschkeCombettes11}.). So $d(x,H_{j})=d(x,C_{j})$,
and \eqref{eq:gen-lin-metric-ineq} reduces to the well known linear
metric inequality (which is sometimes referred to as linear regularity)
for collections of convex sets. A local version of the linear metric
inequality is often defined for the local convergence of projection
algorithms. But for this paper, we shall use the global version defined
above to simplify our analysis.

\subsection{\label{sub:dual-active-set-QP}Using quadratic programming to accelerate
projection algorithms}

One way to accelerate projection algorithms for solving the SIP \eqref{eq:SIP}
is to collect the halfspaces produced by the projection process and
use a quadratic program (QP) to project onto the intersection of these
halfspaces. See \cite{cut_Pang12} for more on this acceleration.
The material in this subsection can be skipped in understanding the
main contributions of the paper. But we feel that a brief mention
of this acceleration can be useful because it shows how developments
in projection methods for solving the SIP can be incorporated in the
algorithms of this paper. 

A QP can be written as 
\begin{eqnarray*}
 & \min & \frac{1}{2}\|x-x_{0}\|^{2}\\
 & \mbox{s.t.} & x\in\cap_{i=1}^{m}H_{i},
\end{eqnarray*}
where $H_{i}$ are halfspaces. If $m$ is small, then the optimal
solution can be found with an efficient QP algorithm once the QR factorization
of the normals of $H_{i}$ are obtained. 

If $m$ is large, then trying to solve the QP would defeat the purpose
of using first order algorithms. We suggest using the dual active
set QP algorithm of \cite{Goldfarb_Idnani}. The $k$th iteration
updates a solution $x_{k}$ and an active set $S_{k}\subset\{1,\dots,m\}$
such that $x_{k}\in\partial H_{i}$ for all $i\in S_{k}$ and $x_{k}=P_{\cap_{i\in S_{k}}H_{i}}(x_{0})$.
The algorithm of \cite{Goldfarb_Idnani} has two advantages:
\begin{enumerate}
\item Each iteration involves relatively cheap updates of the QR factorization
of the normals of the active constraints and solving at most $|S_{k}|$
linear systems of size at most $|S_{k}|$. 
\item The distance $d(x_{0},\cap_{i\in S_{k}}H_{i})=\|x_{0}-x_{k}\|$ is
strictly increasing till it reaches $d(x_{0},\cap_{i=1}^{m}H_{i})$. 
\end{enumerate}
So if the QP were not solved to optimality, each iteration gives a
halfspace $\bar{H}_{k}=\{x:\langle x_{0}-x_{k},x-x_{k}\rangle\leq0\}$
such that $\bar{H}_{k}\supset\cap_{i=1}^{m}H_{i}$ and $d(x_{0},\bar{H}_{k})=d(x_{0},\cap_{i\in S_{k}}H_{i})$,
which is strictly increasing by property (2). The size of the active
set $|S_{k}|$ can reduced if some of the halfspaces are aggregated
into a single halfspace, just like in the generalized Haugazeau's
algorithm in Section \ref{sec:gen-Haugazeau}.

To accelerate an alternating projection strategy, the QR factorization
of the normals of the halfspaces containing $x$, (the point where
one projects from) can be used to find a separating halfspace that
is further away from $x$ at the cost of an iteration of the algorithm
in \cite{Goldfarb_Idnani}.

\section{\label{sec:subgrad-alg}A subgradient algorithm for constrained optimization}

In this section, we look at how to adapt \cite[Theorem 3.2.3]{Nesterov_book}
to treat the case where the number of constraints is large. We begin
by describing our algorithm.
\begin{algorithm}
\label{alg:subgradient}(Subgradient algorithm with feasibility updates)
Let $f:\mathbb{R}^{n}\to\mathbb{R}$ and $f_{j}:\mathbb{R}^{n}\to\mathbb{R}$
(where $j\in\{1,\dots,m\}$) be convex functions. Let $Q\subset\mathbb{R}^{n}$
be a closed convex set, and $R>0$ be such that $\|x-y\|\leq R$ for
all $x,y\in Q$. Let 
\begin{eqnarray}
C_{j} & := & \{x:f_{j}(x)\leq0\}\label{eq:C-j-and-C}\\
\mbox{ and }C_{\phantom{j}} & := & \{x:f_{j}(x)\leq0,j=1,\dots,m\}=\cap_{j=1}^{m}C_{j}.\nonumber 
\end{eqnarray}
This algorithm seeks to solve 
\begin{equation}
\min\{f(x):x\in Q,f_{j}(x)\leq0,j=1,\dots,m\}.\label{eq:constrained_pblm}
\end{equation}

01 \textbf{Step 0}. Choose $x_{0}\in Q$ and sequence $\{h_{k}\}_{k=0}^{\infty}$
by 
\[
h_{k}=\frac{R}{\sqrt{k+0.5}}.
\]

02\textbf{ Step 1}: $k$th iteration ($k\geq0$). Use either Step
1A or Step 1B to find $x_{k+1}$:

03\textbf{ Step 1A}. (Supporting halfspace from $x_{k}$). 

04 Find $\bar{g}_{j,k}\in\partial f_{j}(x_{k})$ for all $j\in\{1,\dots,m\}$. 

05 Define the halfspace $H_{j,k}$ by 
\[
H_{j,k}:=\begin{cases}
\{x:f_{j}(x_{k})+\langle\bar{g}_{j,k},x-x_{k}\rangle\leq0\}. & \mbox{ if }f_{j}(x_{k})\geq0\\
\mathbb{R}^{n} & \mbox{ otherwise}
\end{cases}
\]

06 If $\max_{1\leq j\leq m}d(x_{k},H_{j,k})<h_{k}$, then find $g_{k}\in\partial f(x_{k})$
and set 
\begin{equation}
x_{k+1}=P_{Q}\bigg(x_{k}-\frac{h_{k}}{\|g_{k}\|}g_{k}\bigg).\label{eq:add-proj-to-PQ-1}
\end{equation}

07 Otherwise there is a halfspace $H_{k}$ 

08 $\quad$such that $\cap_{j=1}^{m}H_{j,k}\subset H_{k}$ and $d(x_{k},H_{k})\geq h_{k}$.
Set 
\begin{equation}
x_{k+1}=P_{Q}\circ P_{H_{k}}(x_{k}).\label{eq:step-1A-feas-step}
\end{equation}

09\textbf{ Step 1B}. (Alternating projection strategy) 

10 Let $x_{k}^{0}=x_{k}$. 

11 For $j=\{1,\dots,m\}$

\textup{\emph{12 $\quad$Find $\bar{g}_{j,k}\in\partial f_{j}(x_{k}^{j-1})$.}}

\textup{\emph{13 $\quad$Define the halfspace $H_{j,k}$ by 
\[
H_{j,k}:=\begin{cases}
\{x:f(x_{k}^{j-1})+\langle\bar{g}_{j,k},x-x_{k}^{j-1}\rangle\leq0\} & \mbox{ if }f(x_{k}^{j-1})\geq0\\
\mathbb{R}^{n} & \mbox{ otherwise}.
\end{cases}
\]
}}

\textup{\emph{14 $\quad$Set $S_{j,k}$ to be a subset of $\{1,\dots,j\}$
such that $j\in S_{j,k}$.}}

\textup{\emph{15 $\quad$Set $x_{k}^{j}=P_{\cap_{l\in S_{j,k}}H_{l,k}}(x_{k}^{j-1})$.}}

16 End For.

17 If at any point $\sum_{l=1}^{j}\|x_{k}^{l}-x_{k}^{l-1}\|^{2}\geq h_{k}^{2}$,
then set $x_{k+1}=P_{Q}(x_{k}^{j})$. 

18 Otherwise, choose $g_{k}\in\partial f(x_{k}^{m})$ and set 
\begin{equation}
x_{k+1}=P_{Q}\bigg(x_{k}^{m}-\frac{h_{k}}{\|g_{k}\|}g_{k}\bigg).\label{eq:add-proj-to-PQ-2}
\end{equation}

\end{algorithm}
We make a few remarks about Algorithm \ref{alg:subgradient}. Algorithm
\ref{alg:subgradient} is adapted from \cite[Theorem 3.2.3]{Nesterov_book}
so that if $m$ is large, then one may only need to evaluate a few
of $f_{j}(x_{k})$ and $\bar{g}_{j,k}$, where $j\in\{1,\dots,m\}$,
in the $k$th iteration to find $x_{k+1}$. 
\begin{rem}
(Using quadratic programming to accelerate projection algorithms)
The set $S_{j,k}$ in Step 1B can be chosen to be $S_{j}=\{j\}$,
and Step 1B would correspond to an alternating projection strategy.
But if the size $|S_{j,k}|$ is small, then each step can still be
carried out quickly. Depending on the orientation of the sets $C_{j}$
(see \eqref{eq:C-j-and-C}), choosing a larger set $S_{j,k}$ can
accelerate the convergence of the algorithm as the intersection of
more than one of the halfspaces $H_{j,k}$ would be a better approximate
of the set $C$ than a set of the form $C_{j}$. The strategies outlined
in Subsection \ref{sub:dual-active-set-QP} can be applied.
\end{rem}
In order to accelerate convergence, we can take $x_{k+1}=P_{\cap_{l\in S_{k}}H_{l,k}}\circ P_{Q}(y_{k})$
in \eqref{eq:add-proj-to-PQ-1} and \eqref{eq:add-proj-to-PQ-2},
where $S_{k}\subset\{1,\dots,m\}$ and $y_{k}$ is the formula in
$P_{Q}(\cdot)$. A halfspace separating $P_{Q}(y_{k})$ from $\cap_{l=1}^{m}H_{l,k}$
can be found with the strategies in Subsection \ref{sub:dual-active-set-QP}. 
\begin{rem}
(Choices of $x_{k+1}$ in Step 1A) In Step 1A of Algorithm \ref{alg:subgradient},
it is possible that $\max_{1\leq j\leq m}d(x_{k},H_{j,k})<h_{k}$
and there is a halfspace $H_{k}$ such that $\cap_{j=1}^{m}H_{j,k}\subset H_{k}$
and $d(x_{k},H_{k})\geq h_{k}$. The halfspace $H_{k}$ satisfying
the required properties can be found (by the strategies outlined in
Subsection \ref{sub:dual-active-set-QP} for example) before all the
distances $d(x_{k},H_{j,k})$, where $j\in\{1,\dots,m\}$, are evaluated,
so one would carry out the step \eqref{eq:step-1A-feas-step} in such
a case.
\begin{rem}
(Order of evaluating $f_{j}(\cdot)$s) In both Steps 1A and 1B, we
do not have to loop through the functions $\{f_{j}(\cdot)\}_{j=1}^{m}$
in the sequential order. The functions $\{f_{j}(\cdot)\}_{j=1}^{m}$
can be handled in any order that goes through all the indices in $\{1,\dots,m\}$.
If $j\in\{1,\dots,m\}$ is such that $f_{j}(x^{*})<0$ for all optimal
solutions $x^{*}$, then $f_{j}(\cdot)$ shall be evaluated infrequently.
One can also incorporate ideas in \cite{HermanChen08} to find a good
order to cycle through the indices $\{1,\dots,m\}$.
\end{rem}
\end{rem}
Step 1B describes an extended alternating projection procedure to
find a point that is close to $C$. In view of studies in alternating
projections, one is more likely to achieve feasibility by projecting
from the most recently evaluated point $x_{k}^{j}$ instead of $x_{k}$. 
\begin{thm}
\label{thm:subgrad-conv-rate} (Convergence of Algorithm \ref{alg:subgradient})
Consider Algorithm \ref{alg:subgradient}. Let $x^{*}$ be some optimal
solution. Let $f(\cdot)$ be Lipschitz continuous on $B(x^{*},R)$
with constant $M_{1}$ and let $M_{2}$ be 
\[
M_{2}=\max_{1\leq j\leq m}\{\|g\|:g\in\partial f_{j}(x),x\in B(x^{*},R)\}.
\]

(a) If Step 1A was carried throughout, then for any $k\geq3$, there
exists a number $i^{\prime}$, $0\leq i^{\prime}\leq k$ such that
\begin{equation}
f(x_{i^{\prime}})-f^{*}\leq\frac{\sqrt{3}M_{1}R}{\sqrt{k-1.5}}\mbox{ and }\max\{f_{j}(x_{i^{\prime}}):j\in\{1,\dots,m\}\}\leq\frac{\sqrt{3}M_{2}R}{\sqrt{k-1.5}}.\label{eq:subgrad-conv-1}
\end{equation}

(b) Recall the definition of $C$ in \eqref{eq:C-j-and-C}. If Step
1B was carried throughout, $Q=\mathbb{R}^{n}$ and the linear metric
inequality condition is satisfied for some constant $\kappa<\infty$,
then there exists a number $i^{\prime}$, $0\leq i^{\prime}\leq k$
such that 
\begin{equation}
f(x_{i^{\prime}}^{m})-f^{*}\leq\frac{\sqrt{3}M_{1}R}{\sqrt{k-1.5}}\mbox{ and }d(x_{i^{\prime}}^{m},C)\leq\frac{\kappa\sqrt{3m}R}{\sqrt{k-1.5}}.\label{eq:subgrad-conv-2}
\end{equation}
\end{thm}
\begin{proof}
We first prove for Step 1A. Let $k^{\prime}=\left\lfloor k/3\right\rfloor $
and 
\[
I_{k}=\bigg\{ i\in[k^{\prime},\dots,k]:x_{k+1}=P_{Q}\bigg(x_{k}-\frac{h_{k}}{\|g_{k}\|}g_{k}\bigg)\bigg\}.
\]
When $i\notin I_{k}$, we have $\|x_{k+1}-x^{*}\|^{2}\leq\|x_{k}-x^{*}\|^{2}-h_{k}^{2}$
from the 1-attractiveness of the projection operation. When $i\in I_{k}$,
we have 
\begin{eqnarray}
\|x_{i+1}-x^{*}\|^{2} & \leq & \left\Vert \left[x_{i}-\frac{h_{i}}{\|g_{i}\|}g_{i}\right]-x^{*}\right\Vert ^{2}\label{eq:chain-in-I-k}\\
 & \leq & \|x_{i}-x^{*}\|^{2}+h_{i}^{2}-2h_{i}\bigg\langle\frac{g_{i}}{\|g_{i}\|},x_{i}-x^{*}\bigg\rangle.\nonumber 
\end{eqnarray}
Summing up these inequalities for $i\in[k^{\prime},\dots,k]$ gives
\[
\|x_{k+1}-x^{*}\|^{2}\leq\|x_{k^{\prime}}-x^{*}\|^{2}-\sum_{i\in I_{k}}\bigg[2h_{i}\bigg\langle\frac{g_{i}}{\|g_{i}\|},x_{k}-x^{*}\bigg\rangle-h_{i}^{2}\bigg]-\sum_{i\notin I_{k}}h_{i}^{2}.
\]
Let $v_{i}=\langle\frac{g_{i}}{\|g_{i}\|},x_{i}-x^{*}\rangle$. Seeking
a contradiction, assume that $v_{i}\geq h_{i}$ for all $i\in I_{k}$.
Then 
\[
\sum_{i=k^{\prime}}^{k}h_{i}^{2}=\sum_{i\in I_{k}}h_{i}^{2}+\sum_{{i\notin I_{k}\atop i\in[k^{\prime},\dots,k]}}h_{i}^{2}\leq\|x_{k^{\prime}}-x^{*}\|^{2}\leq R^{2},
\]
which gives 
\[
1\geq\frac{1}{R^{2}}\sum_{i=k^{\prime}}^{k}h_{i}^{2}=\sum_{i=k^{\prime}}^{k}\frac{1}{i+0.5}\geq\int_{k^{\prime}}^{k+1}\frac{d\tau}{\tau+0.5}=\ln\frac{2k+3}{2k^{\prime}+1}\geq\ln3.
\]
This is a contradiction. Thus $I_{k}\neq\emptyset$ and there exists
some $i^{\prime}\in I_{k}$ such that $v_{i^{\prime}}<h_{i^{\prime}}$.
Clearly, for this number we have $v_{i^{\prime}}\leq h_{k^{\prime}}$.
Lemma 3.2.1 in \cite{Nesterov_book} shows that $f(x_{i^{\prime}})-f(x^{*})\leq M_{1}\max\{0,v_{i^{\prime}}\}$.
So 
\begin{equation}
f(x_{i^{\prime}})-f(x^{*})\leq M_{1}h_{k^{\prime}},\label{eq:val-ineq}
\end{equation}
which implies the first part of \eqref{eq:subgrad-conv-1}. 

We now prove the second part of \eqref{eq:subgrad-conv-1}. Since
$i^{\prime}\in I_{k}$, we have $d(x_{i^{\prime}},H_{j,i^{\prime}})\leq h_{i^{\prime}}$
for all $j\in\{1,\dots,m\}$. We can calculate that $d(x_{i^{\prime}},H_{j,i^{\prime}})=\frac{f_{j}(x_{i^{\prime}})}{\|\bar{g}_{j,i^{\prime}}\|}$.
Therefore, $\frac{f_{j}(x_{i^{\prime}})}{\|\bar{g}_{j,i^{\prime}}\|}\leq h_{i^{\prime}}$,
which gives $f_{j}(x_{i^{\prime}})\leq\|\bar{g}_{j,i^{\prime}}\|h_{i^{\prime}}\leq M_{2}h_{k^{\prime}}$.
It remains to note that $k^{\prime}\geq\frac{k}{3}-1$, and therefore
$h_{k^{\prime}}\leq\frac{\sqrt{3}R}{\sqrt{k-1.5}}$. This ends the
proof of (a).

We now go on to prove the result if Step 1B had been used throughout
the algorithm. Once again, let $k^{\prime}=\left\lfloor k/3\right\rfloor $
and 
\[
I_{k}=\bigg\{ i\in[k^{\prime},\dots,k]:x_{k+1}=P_{Q}\bigg(x_{k}^{m}-\frac{h_{k}}{\|g_{k}\|}g_{k}\bigg)\bigg\}.
\]
If $i\notin I_{k}$, we still have $\|x_{i+1}-x^{*}\|^{2}\leq\|x_{i}-x^{*}\|^{2}-h_{i}^{2}$.
If $i\in I_{k}$, we have $\|x_{i}^{m}-x^{*}\|\leq\|x_{i}-x^{*}\|$,
and we can use arguments similar to \eqref{eq:chain-in-I-k} to get
\[
\|x_{i+1}-x^{*}\|^{2}\leq\|x_{i}^{m}-x^{*}\|^{2}+h_{i}^{2}-2h_{i}\left\langle \frac{g_{i}}{\|g_{i}\|},x_{i}^{m}-x^{*}\right\rangle .
\]
Define $v_{i}=\langle\frac{g_{i}}{\|g_{i}\|},x_{i}^{m}-x^{*}\rangle$.
By the same reasoning as before, there is some $i^{\prime}\in I_{k}$
such that $v_{i^{\prime}}<h_{i^{\prime}}\leq h_{k^{\prime}}$. By
the reasoning in \eqref{eq:val-ineq}, we have 
\[
f(x_{i^{\prime}}^{m})-f(x^{*})\leq M_{1}h_{k^{\prime}}.
\]
To obtain the other inequality, we note that $d(x_{k}^{j-1},C_{j})\leq\|x_{k}^{j}-x_{k}^{j-1}\|$
for any $j\in\{1,\dots,m\}$. Thus for any $j\in\{1,\dots,m\}$, we
have 
\[
d(x_{i^{\prime}},C_{j})\leq\sum_{l=1}^{m}\|x_{i^{\prime}}^{l}-x_{i^{\prime}}^{l-1}\|\leq\sqrt{m}\sqrt{\sum_{l=1}^{m}\|x_{i^{\prime}}^{l}-x_{i^{\prime}}^{l-1}\|^{2}}<\sqrt{m}h_{i^{\prime}}.
\]
In view of linear metric inequality, we thus have 
\[
d(x_{i^{\prime}},C)\leq\kappa\underset{j\in\{1,\dots,m\}}{\max}d(x_{i^{\prime}},C_{j})\leq\kappa\sqrt{m}h_{i^{\prime}}.
\]
By Fej\'{e}r monotonicity, we have $d(x_{i^{\prime}}^{m},C)\leq\kappa\sqrt{m}h_{i^{\prime}}$.
Like before, $h_{i^{\prime}}\leq h_{k^{\prime}}\leq\frac{\sqrt{3}R}{\sqrt{k-1.5}}$.
Our proof is complete.
\end{proof}

\section{\label{sec:gen-Haugazeau}Convergence rate of generalized Haugazeau's
algorithm }

One method of solving the BAP \eqref{eq:Proj-pblm} is Haugazeau's
algorithm. In this section, we show that a generalized Haugazeau's
algorithm has $O(1/k)$ convergence to the optimal value and $O(1/\sqrt{k})$
convergence to the optimal solution when the linear metric inequality
assumption is satisfied. 
\begin{algorithm}
\label{alg:gen-Haugazeau}(Generalized Haugazeau's algorithm) Let
$f:\mathbb{R}^{n}\to\mathbb{R}$ be in $\mathcal{S}_{\mu,L}^{1,1}$,
where $\mu>0$. For a point $x_{0}$ and several continuous convex
functions $f_{j}:\mathbb{R}^{n}\to\mathbb{R}$, where $j\in\{1,\dots,m\}$,
we want to find the minimizer of $f(\cdot)$ on 
\[
C:=\cap_{j=1}^{m}\{x:f_{j}(x)\leq0\}.
\]
Suppose the linear metric inequality assumption is satisfied. \\
(A choice of $f(\cdot)$ is $f(x):=\frac{1}{2}\|x-x_{0}\|^{2}$, where
$x_{0}$ is some point in $\mathbb{R}^{n}$.) 

01\textbf{ Step 0:} Let $H_{0}^{\circ}=\mathbb{R}^{n}$. 

02 Let $x_{0}$ be the minimizer of $f(\cdot)$ on $\mathbb{R}^{n}$. 

03 For iteration $k=0,1,2,\dots$

04\textbf{ }$\quad$\textbf{Step 1 (Find a halfspace of largest distance
from $x_{k}$):} 

05 $\quad$For $j\in\{1,\dots,m\}$, 

06 $\quad$$\quad$Find $\bar{g}_{j,k}\in\partial f_{j}(x_{k})$ 

07 $\quad$$\quad$Let $H_{j,k}$ be the set 
\[
H_{j,k}:=\begin{cases}
\{x:f_{j}(x_{k})+\langle\bar{g}_{j,k},x-x_{k}\rangle\leq0\} & \mbox{if }f_{j}(x_{k})\geq0\\
\mathbb{R}^{n} & \mbox{ otherwise.}
\end{cases}
\]

08 $\quad$$\quad$Let $\bar{j}\in\{1,\dots,m\}$ be such that $d(x_{k},H_{\bar{j},k})=\max_{j}d(x_{k},H_{j,k})$. 

09 $\quad$$\quad$Let $H_{k}^{+}:=H_{\bar{j},k}$.

10 $\quad$end for

11\textbf{ }$\quad$\textbf{Step 2:} 

12 $\quad$Find the minimizer $x_{k+1}$ of $f(\cdot)$ on $H_{k}^{\circ}\cap H_{k}^{+}$ 

13 $\quad$Let $H_{k+1}^{\circ}=\{x:\langle-f^{\prime}(x_{k+1}),x-x_{k+1}\rangle\leq0\}$. 

14 End for
\end{algorithm}
The halfspace $H_{k+1}^{\circ}$ in Step 2 is designed so that $x_{k+1}$
is the minimizer of $f(\cdot)$ on $H_{k+1}^{\circ}$. Finding the
index $\bar{j}$ such that $d(x_{k},H_{\bar{j},k})=\max_{j}d(x_{k},H_{j,k})$
in Step 1 can be prohibitively expensive if $m$ is large, so the
alternative algorithm below is more reasonable.
\begin{algorithm}
\label{alg:2nd-Haugazeau}(Alternative algorithm) For the same setting
as Algorithm \ref{alg:gen-Haugazeau}, we propose a different algorithm.

01\textbf{ Step 0:} Let $H_{0}^{\circ}$ be $\mathbb{R}^{n}$, and
let $x_{0}$ be the minimizer of $f(\cdot)$ on $\mathbb{R}^{n}$. 

02 Let $k=0$. 

03\textbf{ Step 1:} Set $x_{k}^{0}=x_{k}$ and $H_{0,k}^{\circ}=H_{k}^{\circ}$. 

04 For $j=1,\dots,m$ 

05 $\quad$Find $\bar{g}_{j,k}\in\partial f_{j}(x_{k})$ and set 
\[
H_{j,k}^{+}:=\begin{cases}
\{x:f_{j}(x_{k}^{j-1})+\langle\bar{g}_{j,k},x-x_{k}^{j-1}\rangle\leq0\} & \mbox{ if }f_{j}(x_{k}^{j-1})\geq0\\
\mathbb{R}^{n} & \mbox{ otherwise.}
\end{cases}
\]

06 $\quad$Find the minimizer $x_{k}^{j}$ of $f(\cdot)$ on $H_{j-1,k}^{\circ}\cap H_{j,k}^{+}$. 

07 $\quad$Let $H_{j,k}^{\circ}=\{x:\langle-f^{\prime}(x_{k}^{j}),x-x_{k}^{j}\rangle\leq0\}$.

08 end for 

09\textbf{ Step 2:} Set $x_{k+1}=x_{k}^{m}$

10 Set $k\leftarrow k+1$ and go back to Step 1.\end{algorithm}
\begin{rem}
(Quadratic case in Algorithm \ref{alg:gen-Haugazeau}) We discuss
the particular case when $f(x):=\frac{1}{2}\|x-x_{0}\|^{2}$. In other
words, the optimization problem is the BAP \eqref{eq:Proj-pblm}.
In this case, Algorithm \ref{alg:gen-Haugazeau} reduces to Haugazeau's
algorithm. The problem of minimizing $f(\cdot)$ on the intersection
of two halfspaces is easy enough to solve analytically. Note that
throughout Algorithm \ref{alg:gen-Haugazeau}, the halfspaces $H_{k}^{\circ}$
and $H_{k}^{+}$ contain the set $C$. One can choose to keep more
halfspaces containing $C$ and in Step 2, find the minimizer of $f(\cdot)$
on the intersection of a larger number of halfspaces. The convergence
would be accelerated at the price of solving larger quadratic programs.
One can also apply the strategies in Subsection \ref{sub:dual-active-set-QP}.
\end{rem}
The lemma below is useful in the proof of Theorem \ref{thm:Haugazeau-rate}.
\begin{lem}
\label{lem:conv-rate-seq}(Convergence rate of a sequence) Suppose
$\{\delta_{k}\}_{k}\subset\mathbb{R}$ is a sequence of nonnegative
real numbers satisfying 
\[
\delta_{k+1}\leq\delta_{k}-\epsilon_{1}\left[1-\sqrt{1-\epsilon_{2}\delta_{k}}\right]^{2}+\frac{\alpha}{k^{2}},
\]
where $\epsilon_{1},\epsilon_{2}>0$, $\alpha\geq0$ and $\epsilon_{2}\delta_{1}<1$.
Let $\bar{\epsilon}=\frac{1}{4}\epsilon_{1}\epsilon_{2}^{2}$. 

(a) The convergence of $\{\delta_{k}\}_{k}$ to zero is $O(1/k)$. 

(b) If $\alpha=0$ and $\delta_{k}>0$ for all $k$, then $\{\delta_{k}\}_{k}$
is strictly decreasing, and 
\[
\delta_{k}\leq\frac{1}{\frac{1}{\delta_{0}}+\bar{\epsilon}k}\mbox{ for all }k\geq0,
\]
\end{lem}
\begin{proof}
We first prove (a). Suppose the values $r>0$ and $\tilde{\epsilon}>0$
are small enough so that $\delta_{1}\leq\frac{1}{r}$, $\tilde{\epsilon}\leq\bar{\epsilon}$,
$\delta_{1}\leq\frac{1}{2\tilde{\epsilon}}$ and $r^{2}\bar{\alpha}+r\leq\tilde{\epsilon}$.
Suppose $\delta_{k}\leq\frac{1}{rk}$. Then by the monotonicity of
the function $\delta\mapsto\delta-\tilde{\epsilon}\delta^{2}$ in
the range $\delta\in[0,\frac{1}{2\tilde{\epsilon}}]$ , we have 
\begin{eqnarray*}
\delta_{k+1} & \leq & \delta_{k}-\bar{\epsilon}[\delta_{k}]^{2}+\frac{\bar{\alpha}}{k^{2}}\\
 & \leq & \delta_{k}-\tilde{\epsilon}[\delta_{k}]^{2}+\frac{\bar{\alpha}}{k^{2}}\\
 & \leq & \frac{1}{rk}-\tilde{\epsilon}\frac{1}{r^{2}k^{2}}+\frac{\bar{\alpha}}{k^{2}}\\
 & = & \frac{rk-\tilde{\epsilon}+r^{2}\bar{\alpha}}{r^{2}k^{2}}\\
 & \leq & \frac{rk-r}{r^{2}k^{2}}\\
 & \leq & \frac{1}{r(k+1)}.
\end{eqnarray*}
 Thus $\{\delta_{k}\}_{k}\in O(1/k)$ as needed.

We now prove (b). Like in (a), we have $\delta_{k+1}\leq\delta_{k}-\frac{1}{4}\epsilon_{1}\epsilon_{2}^{2}\delta_{k}^{2}=\delta_{k}-\bar{\epsilon}\delta_{k}^{2}$.
It is clear that $\{\delta_{k}\}_{k}$ is a strictly decreasing sequence
if all terms are positive. Let $\theta_{k}=\frac{1}{\bar{\epsilon}\delta_{k}}$
so that $\delta_{k}=\frac{1}{\bar{\epsilon}\theta_{k}}$. Then 
\[
\delta_{k+1}\leq\delta_{k}-\bar{\epsilon}\delta_{k}^{2}=\frac{1}{\bar{\epsilon}\theta_{k}}-\frac{1}{\bar{\epsilon}\theta_{k}^{2}}=\frac{\theta_{k}-1}{\bar{\epsilon}\theta_{k}^{2}}\leq\frac{1}{\bar{\epsilon}(\theta_{k}+1)}=\frac{1}{\frac{1}{\delta_{k}}+\bar{\epsilon}}.
\]
In other words, $\frac{1}{\delta_{k+1}}\geq\frac{1}{\delta_{k}}+\bar{\epsilon}$.
The conclusion is now straightforward.\end{proof}
\begin{lem}
\label{lem:dist-to-supp-halfspace}(Distance to supporting halfspace)
Suppose $f:\mathbb{R}^{n}\to\mathbb{R}$ is a differentiable strongly
convex function with parameter $\mu$. Let $\bar{x},x\in\mathbb{R}^{n}$
be such that $f(x)<f(\bar{x})$ and $f^{\prime}(\bar{x})\neq0$. Define
the halfspace $\bar{H}$ by $\bar{H}:=\{x:\langle f^{\prime}(\bar{x}),x-\bar{x}\rangle\geq0\}$.
Then the following hold:

(a) $d(x,\bar{H})\geq\frac{1}{\mu}\Big[\|f^{\prime}(\bar{x})\|-\sqrt{\|f^{\prime}(\bar{x})\|^{2}-2\mu[f(\bar{x})-f(x)]}\Big].$

(b) If $\langle f^{\prime}(x),\bar{x}-x\rangle\geq0$, then $\|x-\bar{x}\|\leq\sqrt{\frac{2}{\mu}[f(\bar{x})-f(x)]}.$\end{lem}
\begin{proof}
We first prove (a). We look to solve 
\begin{eqnarray*}
 & \min_{y} & \left\langle \frac{-f^{\prime}(\bar{x})}{\|f^{\prime}(\bar{x})\|},y-\bar{x}\right\rangle \\
 & \mbox{s.t.} & f(\bar{x})+\langle f^{\prime}(\bar{x}),y-\bar{x}\rangle+\frac{\mu}{2}\|y-\bar{x}\|^{2}\leq f(x)
\end{eqnarray*}
For any $y\in\mathbb{R}^{n}$, a lower bound on $f(y)$ is $f(\bar{x})+\langle f^{\prime}(\bar{x}),y-\bar{x}\rangle+\frac{\mu}{2}\|y-\bar{x}\|^{2}$
by strong convexity. Thus if $y$ is such that $\ensuremath{f(y)=f(x)}$,
it must satisfy the constraint of the above problem. The objective
value is $d(y,\bar{H})$. So this optimization problem finds a lower
bound to the distance to the halfspace $\bar{H}$ provided that the
objective value is at most $f(x)$. 

We rewrite the constraint to get 
\begin{eqnarray*}
f(\bar{x})+\langle f^{\prime}(\bar{x}),y-\bar{x}\rangle+\frac{\mu}{2}\|y-\bar{x}\|^{2} & \leq & f(x)\\
\frac{1}{2}\mu\bigg\| y-\bar{x}+\frac{1}{\mu}f^{\prime}(\bar{x})\bigg\|^{2} & \leq & f(x)-f(\bar{x})+\frac{1}{2\mu}\|f^{\prime}(\bar{x})\|^{2}.
\end{eqnarray*}
The feasible set of the optimization problem is thus a ball with center
$\bar{x}-\frac{1}{\mu}f^{\prime}(\bar{x})$. The optimization problem
can be solved analytically by finding the $t$ with smallest absolute
value such that $x=\bar{x}+tf^{\prime}(\bar{x})$ lies on the boundary
of the ball. In other words, 
\begin{eqnarray*}
f(\bar{x})+\langle f^{\prime}(\bar{x}),[\bar{x}+tf^{\prime}(\bar{x})]-\bar{x}\rangle+\frac{\mu}{2}\|[\bar{x}+tf^{\prime}(\bar{x})]-\bar{x}\|^{2} & = & f(x)\\
\frac{1}{2}t^{2}\mu\|f^{\prime}(\bar{x})\|^{2}+t\|f^{\prime}(\bar{x})\|^{2}+f(\bar{x})-f(x) & = & 0.
\end{eqnarray*}
So

\begin{eqnarray*}
t & = & \frac{-\|f^{\prime}(\bar{x})\|^{2}+\|f^{\prime}(\bar{x})\|\sqrt{\|f^{\prime}(\bar{x})\|^{2}-2\mu[f(\bar{x})-f(x)]}}{\mu\|f^{\prime}(\bar{x})\|^{2}}\\
 & = & -\frac{1}{\mu}+\frac{\sqrt{\|f^{\prime}(\bar{x})\|^{2}-2\mu[f(\bar{x})-f(x)]}}{\mu\|f^{\prime}(\bar{x})\|}.
\end{eqnarray*}
The distance of $x$ to $\bar{H}$ is thus at least $\frac{1}{\mu}\big[\|f^{\prime}(\bar{x})\|-\sqrt{\|f^{\prime}(\bar{x})\|^{2}-2\mu[f(\bar{x})-f(x)]}\big]$
as needed, which concludes the proof of (a).

Next, we prove (b). By strong convexity and the given assumption,
we have 
\[
f(\bar{x})\geq f(x)+\langle f^{\prime}(x),\bar{x}-x\rangle+\frac{\mu}{2}\|x-\bar{x}\|^{2}\geq f(x)+\frac{\mu}{2}\|x-\bar{x}\|^{2}.
\]
A rearrangement of the above gives the conclusion we need.
\end{proof}
Before we prove Theorem \ref{thm:Haugazeau-rate}, we need the following
definition.
\begin{defn}
(Triangular property) Consider the function $f_{j}:\mathbb{R}^{n}\to\mathbb{R}$
in Algorithm \ref{alg:2nd-Haugazeau} for some $j\in\{1,\dots,m\}$.
We say that $f_{j}:\mathbb{R}^{n}\to\mathbb{R}$ has the \emph{triangular
property }if for and $y,z\in\mathbb{R}^{n}$ and any $g_{y}\in\partial f_{j}(y)$
and $g_{z}\in\partial f_{j}(z)$, we have 
\begin{equation}
d(y,H_{y})\leq\|y-z\|+d(z,H_{z}),\label{eq:distance-ppty}
\end{equation}
where 
\[
H_{y}:=\begin{cases}
\{x:f_{j}(y)+\langle g_{y},x-y\rangle\leq0\} & \mbox{ if }f_{j}(y)>0\\
\mathbb{R}^{n} & \mbox{ if }f_{j}(y)\leq0,
\end{cases}
\]
and $H_{z}$ is defined similarly. 
\end{defn}
If $f_{j}:\mathbb{R}^{n}\to\mathbb{R}$ is defined by $f_{j}(x)=d(x,C_{j})$
for a closed convex set $C_{j}\subset\mathbb{R}^{n}$, then $g_{y}=\frac{y-P_{C_{j}}(y)}{d(y,C_{j})}$,
$g_{z}=\frac{z-P_{C_{j}}(z)}{d(z,C_{j})}$, $d(y,H_{y})=d(y,C_{j})$
and $d(z,H_{z})=d(z,C_{j})$, so \eqref{eq:distance-ppty} obviously
holds. However, the triangular property need not hold for any convex
function.
\begin{example}
(Failure of triangular property) Let $f_{j}:\mathbb{R}\to\mathbb{R}$
be defined by $f_{j}(x)=\max\{x,2x-1\}$. If $y=0.9$ and $z=1.1$,
we can check that $d(y,H_{y})=0.9$, $d(z,H_{z})=0.6$ and $\|y-z\|=0.2$,
which means that \eqref{eq:distance-ppty} cannot hold.
\end{example}
We now prove the convergence of Algorithms \ref{alg:gen-Haugazeau}
and \ref{alg:2nd-Haugazeau}. 
\begin{thm}
\label{thm:Haugazeau-rate}(Convergence rate of Algorithm \ref{alg:gen-Haugazeau})
Consider the setting in Algorithm \ref{alg:gen-Haugazeau}. Suppose
the linear metric inequality is satisfied. Let $x^{*}$ be the optimal
solution to $\min\{f(x):x\in C\}$, and assume that $f^{\prime}(x^{*})\neq0$. 

(a) In Algorithm \ref{alg:gen-Haugazeau}, the convergence of $\{f(x_{k})\}_{k=1}^{\infty}$
to $f(x^{*})$ satisfies 
\[
f(x^{*})-f(x_{k})\leq\frac{1}{\frac{1}{f(x^{*})-f(x_{0})}+\bar{\epsilon}k},
\]
where $\bar{\epsilon}=\frac{\mu}{2\kappa^{2}\|f^{\prime}(x^{*})\|^{3}}$,
and the convergence of $\{x_{k}\}_{k=1}^{\infty}$ to $x^{*}$ satisfies
\begin{equation}
\|x^{*}-x_{k}\|\leq\sqrt{\frac{2}{\mu}[f(x^{*})-f(x_{k})]}.\label{eq:conv-to-optimizer}
\end{equation}
Thus the convergence of $\{f(x_{k})\}_{k=1}^{\infty}$ to $f(x^{*})$
is $O(1/k)$, and the convergence of $\{x_{k}\}_{k=1}^{\infty}$ to
$x^{*}$ is $O(1/\sqrt{k})$. 

(b) Suppose in addition that the triangular property holds. In Algorithm
\ref{alg:2nd-Haugazeau}, the convergence of $\{f(x_{k})\}_{k=1}^{\infty}$
to $f(x^{*})$ satisfies 
\[
f(x^{*})-f(x_{k})\leq\frac{1}{\frac{1}{f(x^{*})-f(x_{0})}+\frac{\bar{\epsilon}}{m}k},
\]
where $\bar{\epsilon}$ is the same as in (a), and the convergence
of $\{x_{k}\}_{k=1}^{\infty}$ to $x^{*}$ satisfies \eqref{eq:conv-to-optimizer}.\end{thm}
\begin{proof}
We first prove part (a). Consider the halfspace 
\[
H^{*}:=\{x:\langle-f^{\prime}(x^{*}),x-x^{*}\rangle\leq0\}.
\]
The halfspace $H^{*}$ contains $C$, and contains $x^{*}$ on its
boundary. It is clear that $\{f(x_{k})\}_{k=1}^{\infty}$ is an increasing
sequence such that $\lim_{k\to\infty}f(x_{k})=f(x^{*})$. 

By Lemma \ref{lem:dist-to-supp-halfspace}(a), we have 
\begin{equation}
d(x_{k},H^{*})\geq\frac{1}{\mu}\Big[\|f^{\prime}(x^{*})\|-\sqrt{\|f^{\prime}(x^{*})\|^{2}-2\mu[f(x^{*})-f(x_{k})]}\Big].\label{eq:dist-min-halfspace}
\end{equation}
By linear metric inequality, we can find a separating halfspace from
$x_{k}$ to $C$ that is of distance $\frac{1}{\kappa\mu}[\|f^{\prime}(x^{*})\|-\sqrt{\|f^{\prime}(x^{*})\|^{2}-2\mu[f(x^{*})-f(x_{k})]}]$
from $x_{k}$. Thus 
\[
\|x_{k+1}-x_{k}\|\geq\frac{1}{\kappa\mu}\left[\|f^{\prime}(x^{*})\|-\sqrt{\|f^{\prime}(x^{*})\|^{2}-2\mu[f(x^{*})-f(x_{k})]}\right].
\]
The next iterate $x_{k+1}$ lies in the set $H_{k}^{\circ}\cap H_{k}^{+}$,
so $\langle f^{\prime}(x_{k}),x_{k+1}-x_{k}\rangle\geq0$. By the
$\mu$-strong convexity of $f$, we have 
\begin{eqnarray}
f(x_{k+1}) & \geq & f(x_{k})+\langle f^{\prime}(x_{k}),x_{k+1}-x_{k}\rangle+\frac{\mu}{2}\|x_{k+1}-x_{k}\|^{2}\label{eq:x-k-k-plus-1-chain}\\
 & \geq & f(x_{k})+\frac{\mu}{2}\|x_{k+1}-x_{k}\|^{2}\nonumber \\
\Rightarrow f(x_{k+1})-f(x_{k}) & \geq & \frac{1}{2\kappa^{2}\mu}\Big[\|f^{\prime}(x^{*})\|-\sqrt{\|f^{\prime}(x^{*})\|^{2}-2\mu[f(x^{*})-f(x_{k})]}\Big]^{2}.\nonumber 
\end{eqnarray}
Let $\delta_{k}=f(x^{*})-f(x_{k})$. From the above, we have 
\[
\delta_{k+1}\leq\delta_{k}-\epsilon_{1}\left[1-\sqrt{1-\epsilon_{2}\delta_{k}}\right]^{2},
\]
where $\epsilon_{1}=\frac{\|f^{\prime}(x^{*})\|}{2\kappa^{2}\mu}$
and $\epsilon_{2}=\frac{2\mu}{\|f^{\prime}(x^{*})\|^{2}}$. Applying
Lemma \ref{lem:conv-rate-seq}(b) gives the first part of our result.
Next, note that $x^{*}$ lies in the halfspace through $x_{k}$ with
outward normal $-f^{\prime}(x_{k})$, so this gives $\langle f^{\prime}(x_{k}),x^{*}-x_{k}\rangle\geq0$.
We use Lemma \ref{lem:dist-to-supp-halfspace}(b) to get \eqref{eq:conv-to-optimizer}.

We now prove part (b). Like before, Lemma \ref{lem:dist-to-supp-halfspace}(a)
applies to give \eqref{eq:dist-min-halfspace}. By linear metric inequality
with parameter $\kappa$, there is an index $\bar{j}\in\{1,\dots,m\}$
such that for any $s\in\partial f_{\bar{j}}(x_{k})$, the halfspace
$H_{\bar{j}}:=\{x:f_{\bar{j}}(x_{k})+\langle s,x-x_{k}\rangle\leq0\}$
is such that $d(x_{k},H_{\bar{j}})\geq\bar{d}$, where 
\[
\bar{d}:=\frac{1}{\kappa\mu}\Big[\|f^{\prime}(x^{*})\|-\sqrt{\|f^{\prime}(x^{*})\|^{2}-2\mu[f(x^{*})-f(x_{k})]}\Big].
\]
(Note the difference between $H_{\bar{j}}$ and $H_{j,k}^{+}$.)

Since $x_{k}^{j-1}$ minimizes $f(\cdot)$ on $H_{j,k}^{\circ}$ and
$x_{k}^{j}\in H_{j,k}^{\circ}$, we have $\langle f^{\prime}(x_{k}^{j-1}),x_{k}^{j}-x_{k}^{j-1}\rangle\geq0$.
Therefore, just like in \eqref{eq:x-k-k-plus-1-chain}, we have 
\[
f(x_{k}^{j})-f(x_{k}^{j-1})\geq\langle f^{\prime}(x_{k}^{j-1}),x_{k}^{j}-x_{k}^{j-1}\rangle+\frac{\mu}{2}\|x_{k}^{j}-x_{k}^{j-1}\|^{2}\geq\frac{\mu}{2}\|x_{k}^{j}-x_{k}^{j-1}\|^{2}.
\]
The triangular property implies that $d(x_{k}^{\bar{j}-1},H_{k,\bar{j}}^{+})+\|x_{k}^{\bar{j}-1}-x_{k}^{0}\|\geq d(x_{k}^{0},H_{\bar{j}})\geq\bar{d}$.
Therefore, 
\begin{eqnarray*}
\sum_{j=1}^{\bar{j}}\|x_{k}^{j}-x_{k}^{j-1}\| & \geq & \|x_{k}^{\bar{j}}-x_{k}^{\bar{j}-1}\|+\|x_{k}^{\bar{j}-1}-x_{k}^{0}\|\\
 & = & d(x_{k}^{\bar{j}-1},H_{k,\bar{j}}^{+})+\|x_{k}^{\bar{j}-1}-x_{k}^{0}\|\\
 & \geq & \bar{d}.
\end{eqnarray*}
Then 
\begin{eqnarray*}
f(x_{k}^{\bar{j}})-f(x_{k}^{0}) & \geq & \frac{\mu}{2}\sum_{j=1}^{\bar{j}}\|x_{k}^{j}-x_{k}^{j-1}\|^{2}\\
 & \geq & \frac{\mu}{2\bar{j}}\Bigg[\sum_{j=1}^{\bar{j}}\|x_{k}^{j}-x_{k}^{j-1}\|\Bigg]^{2}\geq\frac{\mu}{2\bar{j}}\bar{d}^{2}\geq\frac{\mu}{2m}\bar{d}^{2}.
\end{eqnarray*}
Let $\delta_{k}:=f(x^{*})-f(x_{k})$. We have $f(x_{k+1})-f(x_{k})\geq f(x_{k}^{\bar{j}})-f(x_{k}^{0})\geq\frac{\mu}{2m}\bar{d}^{2}$,
so 
\begin{eqnarray*}
\delta_{k+1} & \leq & \delta_{k}-\frac{\mu}{2m}\bar{d}^{2}\\
 & = & \delta_{k}-\frac{1}{2\kappa^{2}\mu m}\Big[\|f^{\prime}(x^{*})\|-\sqrt{\|f^{\prime}(x^{*})\|^{2}-2\mu[f(x^{*})-f(x_{k})]}\Big]^{2}\\
 & \leq & \delta_{k}-\frac{\|f^{\prime}(x^{*})\|}{2\kappa^{2}\mu m}\Bigg[1-\sqrt{1-\frac{2\mu}{\|f^{\prime}(x^{*})\|}\delta_{k}}\,\Bigg]^{2}.
\end{eqnarray*}
Applying Lemma \ref{lem:conv-rate-seq}(b) gives us the result we
need. Lemma \ref{lem:dist-to-supp-halfspace}(b) still applies to
give \eqref{eq:conv-to-optimizer}.
\end{proof}
One would expect Algorithm \ref{alg:2nd-Haugazeau} to be better than
Algorithm \ref{alg:gen-Haugazeau} and converge faster than the conservative
estimate of the convergence rate in Theorem \ref{thm:Haugazeau-rate}.

\section{\label{sec:gen-constrained-opt}Constrained optimization with strongly
convex objective}

Consider the strategy of using Algorithm \ref{alg:gen-Haugazeau}
to solve \eqref{eq:main-pblm}, where $f\in\mathcal{S}_{\mu,L}^{1,1}$,
and $\{f_{j}(\cdot)\}_{j=1}^{m}$ satisfies the linear metric inequality.
A difficulty of Algorithm \ref{alg:gen-Haugazeau} is in Steps 0 and
2, where one has to minimize $f(\cdot)$ over the intersection of
two halfspaces. A natural question to ask is whether an approximate
minimizer would suffice, and how much effort is needed to calculate
this approximate solution. In this section, we show how to get around
this difficulty by using steepest descent steps to find an approximate
minimizer of $f(\cdot)$ on the intersection of two halfspaces, leading
to an algorithm that has a convergence rate comparable to Algorithm
\ref{alg:gen-Haugazeau}.

We first recall the constrained steepest descent of functions in $\mathcal{S}_{\mu,L}^{1,1}$
constrained over a simple set and recall its convergence properties
to the minimizer.
\begin{algorithm}
\label{alg:constrained-steepest-descent}(Constrained gradient algorithm)
Consider $f:\mathbb{R}^{n}\to\mathbb{R}$ in $\mathcal{S}_{\mu,L}^{1,1}$
and a closed convex set $Q\subset\mathbb{R}^{n}$. Choose $x_{0}\in Q$.
The constrained gradient algorithm to solve 
\[
\min\{f(x):x\in Q\}
\]
runs as follows:

At iteration $k$ (where $k\geq0$), $x_{k+1}=x_{k}-hP_{Q}(x_{k}-\frac{1}{L}f^{\prime}(x_{k}))$.
\end{algorithm}
Associated with the steepest descent algorithm is the following result.
See for example \cite[Theorem 2.2.8]{Nesterov_book}.
\begin{thm}
\label{thm:convergence-steepest-descent}(Linear convergence to optimizer
of gradient algorithm) Consider Algorithm \ref{alg:constrained-steepest-descent}.
Let $x^{*}$ be the minimizer. If $h=1/L$, then 
\[
\|x_{k+1}-x^{*}\|^{2}\leq\left(1-\frac{\mu}{L}\right)\|x_{k}-x^{*}\|^{2}.
\]

\end{thm}
Actually, the optimal algorithm of \cite{Nesterov_book} gives a better
ratio of $(1-\sqrt{\frac{u}{L}})$ in place of $(1-\frac{u}{L})$,
but the ratio $(1-\frac{u}{L})$ is sufficient for our purposes. In
problems whose main difficulty is in handling a large number of constraints
rather than the dimension of the problem, algorithms which converge
faster than first order algorithms can be used instead. A different
choice of algorithm would however not affect our subsequent analysis. 

We now state our algorithm.
\begin{algorithm}
\label{alg:general-obj-optim}(Constrained optimization with objective
in $\mathcal{S}_{\mu,L}^{1,1}$) Consider $f:\mathbb{R}^{n}\to\mathbb{R}$
in $\mathcal{S}_{\mu,L}^{1,1}$, and let $f_{j}:\mathbb{R}^{n}\to\mathbb{R}$,
where $j\in\{1,\dots,m\}$, be linearly regular convex functions.

\textbf{01 Separating halfspace procedure:}

02 For a point $x\in\mathbb{R}^{n}$, a separating halfspace $H^{+}$
is found as follows:

03 $\quad$For $j\in\{1,\dots,m\}$, 

04 $\quad$$\quad$Find some $\bar{g}_{j}\in\partial f_{j}(x)$

05 $\quad$$\quad$Let 
\[
H_{j}:=\begin{cases}
\{x^{\prime}:f_{j}(x)+\langle g_{j},x^{\prime}-x\rangle\leq0\} & \mbox{ if }f_{j}(x)\geq0\\
\mathbb{R}^{n} & \mbox{ otherwise.}
\end{cases}
\]

06 $\quad$end for

07 $\quad$Let $H^{+}=H_{\bar{j}}$, where $\bar{j}=\arg\max_{1\leq j\leq m}d(x,H_{j})$.

$\quad$

\textbf{01 Main Algorithm:}

02 Let $H_{1}^{\circ}=\mathbb{R}^{n}$ and $H_{1}^{+}=\mathbb{R}^{n}$
and let $x_{1}^{\circ}$ be a starting iterate. Let $\alpha>0$.

03 For $k=1,\dots$

04 $\quad$Let $x_{k}^{*}$ be the minimizer of $f(\cdot)$ on $H_{k}^{\circ}\cap H_{k}^{+}$

05 $\quad$Starting from $x_{k}^{\circ}$, perform constrained gradient
iterations (Algorithm \ref{alg:constrained-steepest-descent})

06 $\quad$$\quad$for solving $\min\{f(x):x\in H_{k}^{\circ}\cap H_{k}^{+}\}$
to find $x_{k}$ such that 

07 $\quad$$\quad$(1) $\|x_{k}-x_{k}^{*}\|\leq\frac{\alpha}{k^{2}}$,
and

08 $\quad$$\quad$(2) $d(x_{k},H_{k+1}^{+})\geq2\|x_{k}-x_{k}^{*}\|$,
where $H_{k+1}^{+}$ is a halfspace obtained

09 $\quad$$\quad$$\phantom{\mathbf{(2)}}$ from the separating halfspace
procedure with input $x_{k}$.

10 $\quad$If $H_{k}^{\circ}\neq\mathbb{R}^{n}$ and $H_{k}^{+}\neq\mathbb{R}^{n}$
(i.e., both $H_{k}^{\circ}$ and $H_{k}^{+}$ are proper halfspaces)

11 $\quad$\textbf{Combine halfspaces $H_{k}^{\circ}$ and $H_{k}^{+}$
to form one halfspace $H_{k+1}^{\circ}$:}

12 $\quad$$\quad$If $\partial H_{k}^{\circ}\cap\partial H_{k}^{+}=\emptyset$
or $d(x_{k},\partial H_{k}^{\circ}\cap\partial H_{k}^{+})>\frac{\alpha}{k^{2}}$ 

13 $\quad$$\quad$$\quad$Let $H_{k+1}^{\circ}=H_{k}^{+}$

14 $\quad$$\quad$else

15 $\quad$$\quad$$\quad$Project $-f^{\prime}(x_{k})$ onto $\cone(\{n_{k}^{\circ},n_{k}^{+}\})$
to get $v\in\mathbb{R}^{n}$, 

16 $\quad$$\quad$$\quad$$\quad$where $n_{k}^{\circ}$ and $n_{k}^{+}$
are the outward normal vectors of $H_{k}^{\circ}$ and $H_{k}^{+}$.

17 $\quad$$\quad$$\quad$Project $x_{k}$ onto $\partial H_{k}^{\circ}\cap\partial H_{k}^{+}$
to get $\tilde{x}_{k}$.

18 $\quad$$\quad$$\quad$Let $H_{k+1}^{\circ}:=\{x:\langle v,x-\tilde{x}_{k}\rangle\leq0\}$

19 $\quad$$\quad$end if

20 $\quad$else

21 $\quad$$\quad$ Let $H_{k+1}^{\circ}=H_{k}^{+}$

22 $\quad$end if

23 end for
\end{algorithm}
Algorithm \ref{alg:general-obj-optim} is actually a two stage process.
We refer to the iterations of finding $\{x_{k}\}$, $\{H_{k}^{\circ}\}$
and $\{H_{k}^{+}\}$ as the outer iterations, and the iterations of
the constrained steepest descent algorithm to find $x_{k}$ as the
inner iterations. 

We didn't mention the starting iterate $x_{k}^{\circ}$ for the constrained
steepest descent algorithm. We can let $x_{k}^{\circ}=x_{k-1}$ for
$k>1$, but setting $x_{k}^{\circ}=x_{1}^{\circ}$ is sufficient for
our analysis.

Throughout the algorithm, the points $x_{k}^{*}$ are not found explicitly.
The distance $\|x_{k}-x_{k}^{*}\|$ can be estimated from Theorem
\ref{thm:convergence-steepest-descent}.

We make a few remarks about Algorithm \ref{alg:general-obj-optim}.
At the beginning of the algorithm, the sets $H_{1}^{+}$ and $H_{1}^{\circ}$
equal $\mathbb{R}^{n}$, but after some point, they become proper
halfspaces. It is clear from the construction of $H_{k+1}^{\circ}$
that $H_{k}^{\circ}\cap H_{k}^{+}\subset H_{k+1}^{\circ}$, and that
$H_{k}^{+}$ are designed so that $C\subset H_{k}^{+}$, so $C\subset H_{k}^{\circ}$.

Assume that $H_{k}^{\circ}$ and $H_{k}^{+}$ are proper halfspaces.
Then the sets $\partial H_{k}^{\circ}$ and $\partial H_{k}^{+}$
are affine spaces with codimension 1. In order for $\partial H_{k}^{\circ}\cap\partial H_{k}^{+}=\emptyset$,
the normals of the halfspaces have to be in the same direction. The
condition 
\[
d(x_{k},\partial H_{k}^{\circ}\cap\partial H_{k}^{+})>\frac{\alpha}{k^{2}}\geq\|x_{k}-x_{k}^{*}\|
\]
 implies that $x_{k}^{*}$ cannot be on $\partial H_{k}^{\circ}\cap\partial H_{k}^{+}$,
so $x_{k}^{*}$ has to lie only on either $\partial H_{k}^{\circ}$
or $\partial H_{k}^{+}$, but not both. By the workings of Algorithm
\ref{alg:general-obj-optim}, $x_{k}^{*}$ cannot lie on $\partial H_{k}^{\circ}$,
and must lie on $\partial H_{k}^{+}$. This explains why $H_{k+1}^{\circ}=H_{k}^{+}$
in the situations specified. 
\begin{thm}
\label{thm:conv-rates-gen-alg}(Convergence of Algorithm \ref{alg:general-obj-optim})
Consider Algorithm \ref{alg:general-obj-optim}. We have 

(1) $\{f(x^{*})-f(x_{k}^{*})\}_{k=1}^{\infty}\in O(1/k)$. 

(2) $\{\|x_{k}^{*}-x^{*}\|\}\in O(1/\sqrt{k})$.

(3) $\{f(x^{*})-f(x_{k})\}\in O(1/k)$, and $\{\|x_{k}-x^{*}\|\}\in O(1/\sqrt{k})$.\end{thm}
\begin{proof}
From strong convexity, we have 
\begin{equation}
f(x_{k+1}^{*})-f(x_{k}^{*})-\frac{\mu}{2}\|x_{k+1}^{*}-x_{k}^{*}\|^{2}\geq\langle f^{\prime}(x_{k}^{*}),x_{k+1}^{*}-x_{k}^{*}\rangle.\label{eq:master-ineq}
\end{equation}
Recall that $n_{k}^{\circ}$ and $n_{k}^{+}$ are the outward normals
of the halfspaces $H_{k}^{\circ}$ and $H_{k}^{+}$ respectively.
The optimality conditions imply that $-f^{\prime}(x_{k}^{*})\in\cone(\{n_{k}^{\circ},n_{k}^{+}\})$. 

When $k=1$ or $2$, the halfspace $H_{k}^{\circ}$ equals $\mathbb{R}^{n}$,
so $H_{k+1}^{\circ}$ equals $H_{k}^{+}$. In the case when $d(x_{k},\partial H_{k}^{\circ}\cap\partial H_{k}^{+})>\frac{\alpha}{k^{2}}$,
we also have $H_{k+1}^{\circ}=H_{k}^{+}$. In these cases, $H_{k+1}^{\circ}=\{x:\langle f^{\prime}(x_{k}^{*}),x-x_{k}^{*}\rangle\geq0\}$.
The inequality \eqref{eq:master-ineq} reduces to 
\begin{equation}
f(x_{k+1}^{*})-f(x_{k}^{*})-\frac{\mu}{2}\|x_{k+1}^{*}-x_{k}^{*}\|^{2}\geq0.\label{eq:master-1}
\end{equation}
We now address the other case where $H_{k}^{\circ}$ and $H_{k}^{+}$
are both proper halfspaces and $d(x_{k},\partial H_{k}^{\circ}\cap\partial H_{k}^{+})\leq\frac{\alpha}{k^{2}}$.
Since $v=P_{\scriptsize\cone(\{n_{k}^{\circ},n_{k}^{+}\})}(-f^{\prime}(x_{k}))$
and $-f^{\prime}(x_{k}^{*})=P_{\scriptsize\cone(\{n_{k}^{\circ},n_{k}^{+}\})}(-f^{\prime}(x_{k}^{*}))$,
the nonexpansivity of the projection onto the convex set $\cone(\{n_{k}^{\circ},n_{k}^{+}\})$
and the assumption that $f^{\prime}(\cdot)$ is Lipschitz with parameter
$L$ gives us 
\begin{equation}
\|f^{\prime}(x_{k}^{*})-(-v)\|\leq\|f^{\prime}(x_{k}^{*})-f^{\prime}(x_{k})\|\leq L\|x_{k}^{*}-x_{k}\|.\label{eq:first-chain}
\end{equation}
The halfspace $H_{k+1}^{\circ}$ equals $\{x:\langle v,x-\tilde{x}_{k}\rangle\leq0\}$.
We must have $x_{k+1}^{*}\in H_{k+1}^{\circ}$, which gives 
\begin{equation}
\langle-v,x_{k+1}^{*}-\tilde{x}_{k}\rangle\geq0.\label{eq:second-chain}
\end{equation}
Before we prove \eqref{eq:master-2}, we note that 
\begin{equation}
\|x_{k}-\tilde{x}_{k}\|\leq d(x_{k},\partial H_{k}^{\circ}\cap\partial H_{k}^{+})\leq\frac{\alpha}{k^{2}},\label{eq:third-chain}
\end{equation}
and that $\|x_{k}-x_{k}^{*}\|\leq\frac{\alpha}{k^{2}}$ is a requirement
in Algorithm \ref{alg:general-obj-optim}. We continue with the arithmetic
in \eqref{eq:master-ineq} to get 
\begin{eqnarray}
 &  & f(x_{k+1}^{*})-f(x_{k}^{*})-\frac{\mu}{2}\|x_{k+1}^{*}-x_{k}^{*}\|^{2}\label{eq:master-2}\\
 & \geq & \langle f^{\prime}(x_{k}^{*}),x_{k+1}^{*}-x_{k}^{*}\rangle\nonumber \\
 & = & \langle-v,x_{k+1}^{*}-\tilde{x}_{k}\rangle+\langle-v,\tilde{x}_{k}-x_{k}^{*}\rangle+\langle f^{\prime}(x_{k}^{*})+v,x_{k+1}^{*}-x_{k}^{*}\rangle\nonumber \\
 & \geq & 0-\|v\|\|\tilde{x}_{k}-x_{k}^{*}\|-\|f^{\prime}(x_{k}^{*})+v\|\|x_{k+1}^{*}-x_{k}^{*}\|\mbox{ (by }\eqref{eq:second-chain}\mbox{)}\nonumber \\
 & \geq & -\|v\|[\|x_{k}-x_{k}^{*}\|+\|x_{k}-\tilde{x}_{k}\|]-L\|x_{k}-x_{k}^{*}\|\|x_{k+1}^{*}-x_{k}^{*}\|\mbox{ (by }\eqref{eq:first-chain}\mbox{)}\nonumber \\
 & \geq & -\frac{\alpha}{k^{2}}[2\|v\|+L\|x_{k+1}^{*}-x_{k}^{*}\|]\mbox{ (by }\eqref{eq:third-chain}\mbox{).}\nonumber 
\end{eqnarray}
Next, since $v$ is the projection of $f^{\prime}(x_{k})$ onto $\cone(\{n_{k}^{\circ},n_{k}^{+}\})$,
which is a convex set containing the origin, we have $\|v\|\leq\|f^{\prime}(x_{k})\|$.
Note that $\|x_{k}-x_{k}^{*}\|\leq\frac{\alpha}{k}^{2}\leq\alpha$.
So 
\begin{eqnarray*}
\|v\| & \leq & \|f^{\prime}(x_{k})\|\\
 & \leq & \|f^{\prime}(x_{k}^{*})\|+\|f^{\prime}(x_{k})-f^{\prime}(x_{k}^{*})\|\\
 & \leq & \|f^{\prime}(x_{k}^{*})\|+L\|x_{k}-x_{k}^{*}\|\\
 & \leq & \max\{\|f^{\prime}(x)\|:f(x)\leq f(x^{*})\}+L\alpha.
\end{eqnarray*}
Since $f(x_{k}^{*})\leq f(x^{*})$ and $f(x_{k+1}^{*})\leq f(x^{*})$,
the strong convexity of $f(\cdot)$ implies that $x_{k}^{*}$ and
$x_{k+1}^{*}$ both lie in a bounded set. (See Lemma \ref{lem:est-x-star-k}.)
Therefore, there is a constant $\bar{\alpha}>0$ such that $\alpha[L\|x_{k+1}^{*}-x_{k}^{*}\|+2\|v\|]\leq\bar{\alpha}$.
Continuing from \eqref{eq:master-2}, we have 
\begin{equation}
f(x_{k+1}^{*})-f(x_{k}^{*})-\frac{\mu}{2}\|x_{k+1}^{*}-x_{k}^{*}\|^{2}\geq-\frac{\bar{\alpha}}{k^{2}}.\label{eq:master-3}
\end{equation}

Since $x_{k+1}^{*}\in H_{k+1}^{+}$, we have $\|x_{k+1}^{*}-x_{k}^{*}\|\geq d(x_{k}^{*},H_{k+1}^{+})$.
We now prove that $d(x_{k}^{*},H_{k+1}^{+})\geq\frac{1}{3\kappa}d(x_{k}^{*},C)$.
We get $\frac{1}{\kappa}d(x_{k},C)\leq d(x_{k},H_{k+1}^{+})$ from
linear metric inequality. From $2\|x_{k}-x_{k}^{*}\|\leq d(x_{k},H_{k+1}^{+})$
and the triangular inequality $d(x_{k},H_{k+1}^{+})\leq\|x_{k}-x_{k}^{*}\|+d(x_{k}^{*},H_{k+1}^{+})$,
we have $\|x_{k}-x_{k}^{*}\|\leq d(x_{k}^{*},H_{k+1}^{+})$. Together
with the fact that $\kappa\geq1$, we have 
\begin{eqnarray*}
\frac{1}{\kappa}d(x_{k}^{*},C) & \leq & \frac{1}{\kappa}\|x_{k}-x_{k}^{*}\|+\frac{1}{\kappa}d(x_{k},C)\\
 & \leq & \|x_{k}-x_{k}^{*}\|+d(x_{k},H_{k+1}^{+})\\
 & \leq & 2\|x_{k}-x_{k}^{*}\|+d(x_{k}^{*},H_{k+1}^{+})\\
 & \leq & 3d(x_{k}^{*},H_{k+1}^{+}).
\end{eqnarray*}
We then have 
\begin{eqnarray}
\|x_{k+1}^{*}-x_{k}^{*}\| & \geq & d(x_{k}^{*},H_{k+1}^{+})\label{eq:x-star-diff-LB}\\
 & \geq & \frac{1}{3\kappa}d(x_{k}^{*},C)\nonumber \\
 & \geq & \frac{1}{3\kappa\mu}\Big[\|f^{\prime}(x^{*})\|-\sqrt{\|f^{\prime}(x^{*})\|^{2}-2\mu[f(x^{*})-f(x_{k}^{*})]}\Big].\nonumber 
\end{eqnarray}
(The third inequality comes from Lemma \ref{lem:dist-to-supp-halfspace}(a).) 

Let $\epsilon_{1}=\frac{\|f^{\prime}(x^{*})\|}{3\kappa\mu}$ and $\epsilon_{2}=\frac{2\mu}{\|f^{\prime}(x^{*})\|}$.
Putting \eqref{eq:x-star-diff-LB} into formulas \eqref{eq:master-1}
and \eqref{eq:master-3} gives 
\begin{eqnarray*}
 &  & f(x^{*})-f(x_{k+1}^{*})\leq f(x^{*})-f(x_{k}^{*})-\frac{\mu\epsilon_{1}}{2}\bigg[1-\sqrt{1-\epsilon_{2}[f(x^{*})-f(x_{k}^{*})]}\bigg]^{2}+\frac{\bar{\alpha}}{k^{2}}\\
 & \Rightarrow & \delta_{k+1}^{*}\leq\delta_{k}^{*}-\frac{\mu\epsilon_{1}}{2}\bigg[1-\sqrt{1-\epsilon_{2}\delta_{k}^{*}}\bigg]^{2}+\frac{\bar{\alpha}}{k^{2}},
\end{eqnarray*}
where $\delta_{k}^{*}:=f(x^{*})-f(x_{k}^{*})$. Part (1) now follows
from Lemma \ref{lem:conv-rate-seq}(a).

The optimality conditions on $x_{k}^{*}$ implies that the point $x^{*}$
must lie in the halfspace $\{x:\langle f^{\prime}(x_{k}^{*}),x-x_{k}^{*}\rangle\geq0\}$.
Lemma \ref{lem:dist-to-supp-halfspace}(b) then implies 
\[
\|x^{*}-x_{k}^{*}\|\leq\sqrt{\frac{2}{\mu}[f(x^{*})-f(x_{k}^{*})]}.
\]
The claim in (2) follows immediately from the above inequality and
(1).

To see (3), note that since $f(\cdot)$ is locally Lipschitz at $x^{*}$,
we have 
\[
\limsup_{k\to\infty}\frac{|f(x_{k})-f(x_{k}^{*})|}{\|x_{k}-x_{k}^{*}\|}<\infty.
\]
Since $\{\|x_{k}-x_{k}^{*}\|\}\in O(1/k^{2})$, we have $\{|f(x_{k})-f(x_{k}^{*})|\}\in O(1/k^{2})$.
Next, since $\{|f(x^{*})-f(x_{k}^{*})|\}\in O(1/k)$, we have $\{|f(x^{*})-f(x_{k})|\}\in O(1/k)$
as needed. The other inequality $\{\|x_{k}-x^{*}\|\}\in O(1/\sqrt{k})$
can also be proved with these steps.\end{proof}
\begin{lem}
\label{lem:est-x-star-k}(Estimate of $x_{k}^{*}$) In Algorithm \ref{alg:general-obj-optim},
the points $x_{k}^{*}$ satisfy 
\[
\left\Vert x_{k}^{*}-x^{*}+\frac{1}{\mu}f^{\prime}(x^{*})\right\Vert \leq\left\Vert \frac{1}{\mu}f^{\prime}(x^{*})\right\Vert .
\]
\end{lem}
\begin{proof}
From the $\mu$-strong convexity of $f(\cdot)$ and the fact that
$f(x_{k}^{*})\leq f(x^{*})$, we have $f(x^{*})+\langle f^{\prime}(x^{*}),x_{k}^{*}-x^{*}\rangle+\frac{\mu}{2}\|x_{k}^{*}-x^{*}\|^{2}\leq f(x_{k}^{*})\leq f(x^{*})$,
from which the conclusion follows.
\end{proof}

\subsection{Computational effort of Algorithm \ref{alg:general-obj-optim}}

We now calculate the amount of computational effort that Algorithm
\ref{alg:general-obj-optim} takes to find an iterate $x_{k}$ such
that $|f(x_{k})-f(x^{*})|\leq\epsilon$. The number of outer iterations
needed to find the iterate $x_{k}$ is, by definition, $k$. It therefore
remains to calculate the number of inner iterations corresponding
to each outer iteration. 

Consider the case when $\|x_{k}^{*}-x^{*}\|$ is small (or even zero)
for the final iteration $k$. Even though it means that the outer
iterations in Algorithm \ref{alg:general-obj-optim} have done well
to allow us to get a good $x_{k}^{*}$ once the required number of
inner iterations are performed, the number of inner iterations needed
to satisfy $d(x_{k},H_{k+1}^{+})\geq2\|x_{k}-x_{k}^{*}\|$ can be
excessively large. In view of this difficulty, we leave out the number
of inner iterations associated with the last outer iterate. Nevertheless,
when $d(x_{k},H_{k+1}^{+})$ and $\|x_{k}-x_{k}^{*}\|$ are small,
we have the following estimates on $f(x^{*})-f(x_{k}^{*})$, $\|x_{k}^{*}-x^{*}\|$,
and hence $|f(x^{*})-f(x_{k})|$ and $\|x_{k}-x^{*}\|$ in \eqref{eq:subsequent-formulas-1}
of Theorem \ref{thm:performance-est} from quantities that are calculated
throughout Algorithm \ref{alg:general-obj-optim}. 
\begin{thm}
\label{thm:performance-est}(Performance estimates) Consider Algorithm
\ref{alg:general-obj-optim}. Let $H^{*}$ be the halfspace $\{x:\langle f^{\prime}(x^{*}),x-x^{*}\rangle\geq0\}$.
We have
\begin{enumerate}
\item $0\leq f(x^{*})-f(x_{k}^{*})\leq\|f^{\prime}(x^{*})\|d(x_{k}^{*},H^{*})$.
\item $\|x_{k}^{*}-x^{*}\|\leq\sqrt{2\|f^{\prime}(x^{*})\|d(x_{k}^{*},H^{*})}$.
\end{enumerate}
Suppose $\{f_{j}(\cdot)\}_{j=1}^{m}$ satisfies $\kappa$ linear metric
inequality and an iterate $x_{k}$ of the minimization subproblem
is such that 
\[
\bar{d}:=[\|x_{k}-x_{k}^{*}\|+\kappa d(x_{k},H_{k+1}^{+})].
\]
Then 
\begin{equation}
d(x_{k}^{*},H^{*})\leq\bar{d}.\label{eq:d-d-bar-ineq}
\end{equation}
Hence if $f(\cdot)$ is Lipschitz with constant $M$ in a neighborhood
$U$ of $x^{*}$ and both $x_{k}$ and $x_{k}^{*}$ lie in $U$, then
\begin{subequations} 
\begin{eqnarray}
 &  & |f(x_{k})-f(x^{*})|\leq\|f^{\prime}(x^{*})\|\bar{d}+M\|x_{k}-x_{k}^{*}\|,\label{eq:subsequent-formulas-1}\\
 & \mbox{ and } & \|x_{k}-x^{*}\|\leq\|x_{k}-x_{k}^{*}\|+\sqrt{2\mu\|f^{\prime}(x^{*})\|\bar{d}}.\label{eq:subsequent-formulas-2}
\end{eqnarray}
\end{subequations}\end{thm}
\begin{proof}
Recall that $x^{*}$ is the solution to the original problem. Since
$f(x_{k}^{*})<f(x^{*})$, then either $d(x_{k}^{*},H^{*})>0$ or $x_{k}^{*}=x^{*}$.
When $x_{k}^{*}=x^{*}$, all the conclusions in our result would be
true, so we only look at the first case. It is clear that $d(x_{k}^{*},H^{*})=\langle-\frac{f^{\prime}(x^{*})}{\|f^{\prime}(x^{*})\|},x_{k}^{*}-x^{*}\rangle$.
By the convexity of $f(\cdot)$, we have 
\begin{equation}
f(x^{*})-f(x_{k}^{*})\leq\langle-f^{\prime}(x^{*}),x_{k}^{*}-x^{*}\rangle=\|f^{\prime}(x^{*})\|d(x_{k}^{*},H^{*}).\label{eq:est-2}
\end{equation}

Next, we find an upper bound for $\|x_{k}^{*}-x^{*}\|$. Lemma \ref{lem:est-x-star-k}
states that $x_{k}^{*}$ lies in a ball with radius $\|\frac{1}{\mu}f^{\prime}(x^{*})\|$,
center $z:=x^{*}-\frac{1}{\mu}f^{\prime}(x^{*})$, and has the point
$x^{*}$ on its boundary. See Figure \ref{fig:circle-fig}. The furthest
point $x$ in this ball from $x^{*}$ that satisfies $d(x,H^{*})\leq d(x_{k}^{*},H^{*})$
has to be such that $d(x,H^{*})=d(x_{k}^{*},H^{*})$ and $x$ being
on the boundary of this ball. 

\begin{figure}[!h]
\includegraphics[scale=0.3]{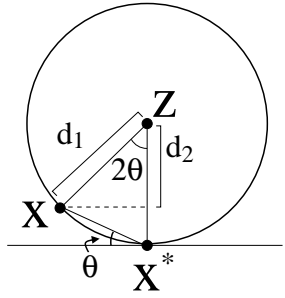}

\caption{\label{fig:circle-fig}Diagram used in the proof of Theorem \ref{thm:performance-est}.
The distances $d_{1}$ and $d_{2}$ equal $\|\frac{1}{\mu}f^{\prime}(x^{*})\|$
and $\|\frac{1}{\mu}f^{\prime}(x^{*})\|-d(x_{k}^{*},H^{*})$ respectively.}

\end{figure}

Finding an upper bound for $\|x_{k}^{*}-x^{*}\|$ is now an easy exercise
in trigonometry. Let $\theta$ be the angle that the line through
$x$ and $x^{*}$ makes with $\partial H^{*}$. We thus have $\angle xzx^{*}=2\theta$.
So $\cos2\theta=\frac{\|f^{\prime}(x^{*})\|-\mu d(x_{k}^{*},H^{*})}{\|f^{\prime}(x^{*})\|}$.
Making use of $\cos2\theta=1-2[\sin\theta]^{2}$, we have 
\[
\sin\theta=\sqrt{\frac{\mu d(x_{k}^{*},H^{*})}{2\|f^{\prime}(x^{*})\|}}.
\]
An upper bound for $\|x_{k}^{*}-x^{*}\|$ is thus $d(x_{k}^{*},H^{*})/\sin\theta$,
so 
\begin{equation}
\|x_{k}^{*}-x^{*}\|\leq d(x_{k}^{*},H^{*})/\sin\theta=\sqrt{2\mu\|f^{\prime}(x^{*})\|d(x_{k}^{*},H^{*})}.\label{eq:get-subseq-2}
\end{equation}
We have 
\begin{equation}
d(x_{k}^{*},H^{*})\leq d(x_{k}^{*},C)\leq\|x_{k}-x_{k}^{*}\|+d(x_{k},C)\leq\|x_{k}-x_{k}^{*}\|+\kappa d(x_{k},H_{k+1}^{+})=\bar{d}.\label{eq:est-1}
\end{equation}
To get \eqref{eq:subsequent-formulas-1}, we make use of \eqref{eq:est-2},
\eqref{eq:d-d-bar-ineq} and the assumption that $f(\cdot)$ is Lipschitz
with constant $M$ to get 
\begin{eqnarray*}
|f(x_{k})-f(x^{*})| & \leq & |f(x^{*})-f(x_{k}^{*})|+|f(x_{k}^{*})-f(x_{k})|\\
 & \leq & \|f^{\prime}(x^{*})\|\bar{d}+M\|x_{k}-x_{k}^{*}\|.
\end{eqnarray*}
Formula \eqref{eq:subsequent-formulas-2} follows easily from \eqref{eq:get-subseq-2}.
\end{proof}
We now calculate the number of inner iterations needed for outer iterations
$j\in\{1,\dots,k-1\}$ so that $|f(x_{k})-f(x^{*})|\leq\epsilon$.
As seen in Theorem \ref{thm:conv-rates-gen-alg}, the convergence
rate of $|f(x_{k})-f(x^{*})|$ is $O(1/k)$, or in other words, $O(1/\epsilon)$
outer iterations would ensure that $|f(x_{k})-f(x^{*})|\leq\epsilon$. 

To ensure that $\|x_{j}-x_{j}^{*}\|\leq\frac{\alpha}{j^{2}}$ for
$j\in\{1,\dots,k-1\}$, we need $O(\log(j^{2}))$ iterations. Since
the number of iterations $k$ is $O(1/\epsilon)$, we need at least
$O(\log(1/\epsilon^{2}))$ iterations for the $(k-1)$th inner subproblem.
We now proceed to find how the condition $d(x_{k},H_{k+1}^{+})\geq2\|x_{k}-x_{k}^{*}\|$
affects the number of inner iterations in each outer iteration. We
have the following inequalities.
\begin{rem}
\label{rem:upp-bdd-on-iter}If $\|x_{k}-x_{k}^{*}\|\leq\frac{1}{3\kappa}d(x_{k}^{*},C)$,
then linear metric inequality, the fact that $\kappa\geq1$, and the
triangular inequality implies 
\begin{eqnarray*}
d(x_{k},H_{k+1}^{+}) & \geq & \frac{1}{\kappa}d(x_{k},C)\\
 & \geq & \frac{1}{\kappa}d(x_{k}^{*},C)-\frac{1}{\kappa}\|x_{k}-x_{k}^{*}\|\\
 & \geq & 3\|x_{k}-x_{k}^{*}\|-\|x_{k}-x_{k}^{*}\|\\
 & \geq & 2\|x_{k}-x_{k}^{*}\|.
\end{eqnarray*}

\end{rem}
Remark \ref{rem:upp-bdd-on-iter} implies that in the $j$th outer
iteration, the number of inner iterations it takes to get $d(x_{j},H_{j+1}^{+})\geq2\|x_{j}-x_{j}^{*}\|$
is at most the number of iterations it takes to get $\|x_{j}-x_{j}^{*}\|\leq\frac{1}{3\kappa}d(x_{j}^{*},C)$. 
\begin{prop}
We continue the discussion of this subsection. Suppose $f(\cdot)$
is Lipschitz with constant $M$. If $d(x_{k}^{*},C)\leq\frac{\epsilon}{\|f^{\prime}(x^{*})\|+M}$,
and $d(x_{k},H_{k+1}^{+})\geq2\|x_{k}-x_{k}^{*}\|$, then $|f(x_{k})-f(x^{*})|\leq\epsilon$. \end{prop}
\begin{proof}
We first prove that $d(x_{k},H_{k+1}^{+})\geq2\|x_{k}-x_{k}^{*}\|$
implies $\|x_{k}-x_{k}^{*}\|\leq d(x_{k}^{*},C)$. We have 
\[
d(x_{k}^{*},C)\geq d(x_{k},C)-\|x_{k}-x_{k}^{*}\|\geq d(x_{k},H_{k+1}^{+})-\|x_{k}-x_{k}^{*}\|\geq\|x_{k}-x_{k}^{*}\|.
\]
Finally, making use of Theorem \ref{thm:performance-est}(1), we have
\begin{eqnarray*}
|f(x_{k})-f(x^{*})| & \leq & |f(x^{*})-f(x_{k}^{*})|+|f(x_{k})-f(x_{k}^{*})|\\
 & \leq & \|f^{\prime}(x^{*})\|d(x_{k}^{*},C)+M\|x_{k}-x_{k}^{*}\|\\
 & \leq & \|f^{\prime}(x^{*})\|d(x_{k}^{*},C)+Md(x_{k}^{*},C)\\
 & \leq & \epsilon.
\end{eqnarray*}

\end{proof}
So we must have $d(x_{j}^{*},C)>\frac{\epsilon}{\|f^{\prime}(x^{*})\|+M}$
for all $j\in\{1,\dots,k-1\}$. For the outer iterations $j\in\{1,\dots,k-1\}$,
Remark \ref{rem:upp-bdd-on-iter} imposes that the number of inner
iterations needs to allow us to get $\|x_{j}-x_{j}^{*}\|\leq\frac{\epsilon}{3\kappa[\|f^{\prime}(x^{*})\|+M]}$.
So the number of inner iterations for outer iterate $j\in\{1,\dots,k-1\}$
needs to be at least $O(\log(1/\epsilon))$, which is less than the
$O(\log(1/\epsilon^{2}))$ obtained earlier. So the total number of
inner iterations in outer iterations $j\in\{1,\dots,k-1\}$ that is
needed to get $|f(x_{k})-f(x^{*})|\leq\epsilon$ is $O(\frac{1}{\epsilon}\log(1/\epsilon^{2}))$.
The corresponding number of inner iterations to get $\|x_{k}-x^{*}\|\leq\epsilon$
can be similarly calculated to be $O(\frac{1}{\epsilon^{2}}\log(1/\epsilon^{4}))$.

\section{\label{sec:effectiveness}Lower bounds on effectiveness of projection
algorithms }

In this section, we derive a lower bound that describes the absolute
rate convergence of first order algorithms where one projects onto
component sets to explore the feasible set. Let $f:\mathbb{R}^{n}\to\mathbb{R}$
be a convex function. When \eqref{eq:main-pblm} is restricted to
the case where $f_{j}(x)$ is an affine function and $Q=\mathbb{R}^{n}$,
we have the following problem 
\begin{eqnarray}
 & \min & f(x)\label{eq:lower-bdd-pblm}\\
 & \mbox{s.t.} & x\in\cap_{j=1}^{m}H_{j}\nonumber \\
 &  & x\in\mathbb{R}^{n},\nonumber 
\end{eqnarray}
where $H_{j}$ are halfspaces. In the case where $m$ and $n$ are
large, only first order algorithms are capable of handling the large
size of the problems. So absolute bounds rather than asymptotic bounds
are more appropriate for the analysis of the speed of convergence
of the algorithms. Motivated by the analysis in \cite{Nesterov_book},
we consider the following algorithm.
\begin{algorithm}
\label{alg:analyze-lower-bdd}(Algorithm to analyze \eqref{eq:lower-bdd-pblm})
Suppose in \eqref{eq:lower-bdd-pblm}, we have the following algorithm.
Let $x_{0}$ be a starting iterate.

01 Set $S_{0}=\emptyset$.

02 For iteration $k\geq1$

03 $\quad$ Find $i_{k}\in\{1,\dots,m\}\backslash S_{k-1}$, and set
$S_{k}=S_{k-1}\cup\{i_{k}\}$.

04 $\quad$ Find objective value $f_{k}$ of $\min\{f(x):x\in\cap_{j\in S_{k}}H_{j}\}$.

05 End for.
\end{algorithm}
A lower bound on the absolute rate of convergence of Algorithm \ref{alg:analyze-lower-bdd}
would give an absolute bound on how algorithms that explore the feasible
set by projection can converge. 

We look in particular at the problem 
\begin{eqnarray}
 & \min & \|e_{1}-x\|_{p}^{p}\label{eq:model-lower-bdd}\\
 & \mbox{s.t.} & \langle[e_{1}+\epsilon e_{j+1}],x\rangle\leq0\mbox{ for }j\in\{1,\dots,n-1\}\nonumber \\
 &  & x\in\mathbb{R}^{n},\nonumber 
\end{eqnarray}
where $\|\cdot\|_{p}$ is the usual $p$ norm defined by $\|x\|_{p}^{p}=\sum_{i=1}^{n}x_{i}^{p}$,
and $e_{i}$ are the elementary vectors with $1$ on the $i$th component
and $0$ everywhere else. We also restrict $p$ to be a positive even
integer, so that the objective function is seen to be convex. 

First, we prove that the constraints satisfy the linear metric inequality.
\begin{prop}
(Linear metric inequality in \eqref{eq:model-lower-bdd}) The sets
in the constraints of \eqref{eq:model-lower-bdd} satisfy the linear
metric inequality. \end{prop}
\begin{proof}
The unit normals of each halfspace is $\frac{1}{\sqrt{1+\epsilon^{2}}}[e_{1}+\epsilon e_{j+1}]$
for each $j\in\{1,\dots,n-1\}$. The distance from the origin to the
convex hull of these unit normals is at least $\frac{1}{\sqrt{1+\epsilon^{2}}}$.
We can make use of the results in \cite{Kruger_06} for example, which
contain what we need. (In fact, much more than what we need.) In the
notation of that paper, linear metric inequality follows from establishing
$\hat{\vartheta}>0$ given $\eta>0$. Theorems 1(i) and 2(ii) there
give $\hat{\vartheta}=\hat{\theta}$ and $\eta\leq\frac{\hat{\theta}}{1-\hat{\theta}}$
respectively. These imply $\hat{\vartheta}=\hat{\theta}\geq\frac{\eta}{1+\eta}>0$. 
\end{proof}
When Algorithm \ref{alg:analyze-lower-bdd} is applied to \eqref{eq:model-lower-bdd},
the symmetry of the problem implies that we can take $S_{k}=\{1,\dots,k\}$.
We now calculate the objective value when $k$ of the constraints
in \eqref{eq:model-lower-bdd} are considered.
\begin{prop}
(Calculating $f_{k}$) In \eqref{eq:model-lower-bdd}, let $p$ be
any positive even integer $p$. Let $f_{k}$ be the optimal value
of \eqref{eq:model-lower-bdd} when only $k$ of the $n-1$ constraints
are taken into account. We have 
\[
f_{k}=\frac{k\theta}{\left[1+(k\theta)^{1/(p-1)}\right]^{p-1}},\mbox{ where }\theta=\epsilon^{-p}.
\]
\end{prop}
\begin{proof}
The function $x\mapsto\|e_{1}-x\|_{p}^{p}$ is seen to be strictly
convex, so there is a unique minimizer. Let $\bar{x}$ be the minimizer
of the $k$th subproblem. The symmetry of the problem implies that
the second to $(k+1)$th component of $\bar{x}$ have the same value,
say $\beta$, and the $(k+2)$th to $n$th component of $\bar{x}$
are zero. Moreover, all the inequality constraints are tight. Let
the first component have the value $\alpha$. We now see that $f_{k}$
equals the objective value of the following problem
\begin{eqnarray*}
f_{k}= & \min_{(\alpha,\beta)} & (1-\alpha)^{p}+k\beta^{p}\\
 & \mbox{s.t.} & \alpha+\epsilon\beta=0.
\end{eqnarray*}
We have $\beta=-\frac{1}{\epsilon}\alpha$. Let $\tilde{\theta}=k\theta$.
We have 
\[
f_{k}=\min_{\alpha}(1-\alpha)^{p}+\tilde{\theta}\alpha^{p}.
\]
The derivative of the above function with respect to $\alpha$ equals
\[
p(1-\alpha)^{p-1}+\tilde{\theta}p\alpha^{p-1}.
\]
Setting the above to zero gives us 
\begin{eqnarray*}
\left(\frac{\alpha-1}{\alpha}\right)^{p-1} & = & \tilde{\theta}\\
\frac{1-\alpha}{\alpha} & = & \tilde{\theta}^{1/p-1}\\
\alpha(1+\tilde{\theta}^{1/p-1}) & = & 1\\
\alpha & = & \frac{1}{1+\tilde{\theta}^{1/p-1}}.
\end{eqnarray*}
This gives us 
\begin{eqnarray*}
f_{k} & = & (1-\alpha)^{p}+\tilde{\theta}\alpha^{p}\\
 & = & \left(\frac{\tilde{\theta}^{1/p-1}}{1+\tilde{\theta}^{1/p-1}}\right)^{p}+\theta\left(\frac{1}{1+\tilde{\theta}^{1/p-1}}\right)^{p}\\
 & = & \frac{\tilde{\theta}^{p/(p-1)}+\tilde{\theta}}{\left[1+\tilde{\theta}^{1/p-1}\right]^{p}}\\
 & = & \frac{\tilde{\theta}}{\left[1+\tilde{\theta}^{1/p-1}\right]^{p-1}}
\end{eqnarray*}
which is what we need. 
\end{proof}
One easy thing to see is that as $k\to\infty$, we have $f_{k}=1$.
This also means that if we make $n$ arbitrarily large, the objective
value converges to $1$. By the binomial theorem, we can calculate
that the leading term of $1-f_{k}$ is 
\[
\frac{(p-1)(k\theta)^{(p-2)/(p-1)}}{\left[1+(k\theta)^{1/(p-1)}\right]^{p-1}}.
\]
This leading term converges to zero at $k\to\infty$ at the rate of
$\Theta(\frac{1}{k^{1/(p-1)}})$, while the other terms converge to
zero at a faster rate. 

Two conclusions can be made with the example presented in this section.
\begin{itemize}
\item The case of $p=2$ gives a convergence rate of $O(\frac{1}{k})$ for
the objective value. This suggests that the methods presented for
strongly convex objective functions in Section \ref{sec:gen-Haugazeau}
are the best possible up to a constant, that the methods in Section
\ref{sec:gen-constrained-opt} are close to the best possible. 
\item The case of $p$ being an arbitrarily large even number gives a convergence
rate of $O(\frac{1}{k^{1/(p-1)}})$. This suggests that if the objective
function is not strongly convex, it would be more sensible to use
the subgradient algorithm (Algorithm \ref{alg:subgradient}) to solve
\eqref{eq:lower-bdd-pblm} instead.
\end{itemize}

\section{\label{sec:Behavior-Haugazeau}Lower bounds on rate of Haugazeau's
algorithm}

In this section, we give two examples in separate subsections to show
the behavior of Haugazeau's algorithm. The first example shows the
$O(1/k)$ convergence rate of the objective value in the case of the
intersection of two halfspaces. This suggests that the convergence
rate of $O(1/k)$ is typical. The second example shows that Haugazeau's
algorithm converges arbitrarily slowly in a convex problem when the
linear metric inequality is not satisfied. 

The lemma below will be used for both examples.
\begin{lem}
\label{lem:lower-bdd-rate}(Lower bound of convergence of a sequence)
Let $p\geq1$. Suppose $\{\alpha_{k}\}_{k=1}^{\infty}$ is a strictly
decreasing sequence of real numbers converging to zero, and there
is some $\gamma>0$ such that $\alpha_{k+1}\geq\alpha_{k}(1-\gamma\alpha_{k}^{p})$
for all $k$. Then we can find a constant $M_{2}\geq0$ such that
$\alpha_{k}\geq\frac{1}{^{p}\sqrt{2p\gamma(k+M_{2})}}$ for all $k\geq1$.\end{lem}
\begin{proof}
By Taylor's Theorem on the function $f(x):=(1-x)^{p}$, we can choose
$M_{2}$ large enough so that 
\begin{equation}
\left[1-\frac{1}{2p(k+M_{2})}\right]^{p}\geq1-\frac{p+1}{2p(k+M_{2})}\mbox{ for all }k\geq0.\label{eq:Taylor}
\end{equation}
We can increase $M_{2}$ if necessary so that 
\begin{enumerate}
\item $(k+M_{2}+1)(k+M_{2}-\frac{p+1}{2p})\geq(k+M_{2})^{2}$ for all $k\geq0$, 
\item $\alpha_{1}\geq\frac{1}{^{p}\sqrt{2p\gamma(1+M_{2})}}$, and 
\item the map $\alpha\mapsto\alpha(1-\gamma\alpha^{p})$ is strictly increasing
in the interval $[0,\frac{1}{^{p}\sqrt{2p\gamma(1+M_{2})}}]$. 
\end{enumerate}
We now show that $\alpha_{i}\geq\frac{1}{^{p}\sqrt{2p\gamma(k+M_{2})}}$
implies $\alpha_{i+1}\geq\frac{1}{^{p}\sqrt{2p\gamma(k+M_{2}+1)}}$
for all $k\geq1$, which would complete our proof. Now, making use
of the fact that $\{\alpha_{k}\}$ is strictly decreasing and (3),
we have 
\[
\alpha_{k+1}\geq\alpha_{k}(1-\gamma\alpha_{k}^{p})\geq\frac{1}{^{p}\sqrt{2p\gamma(k+M_{2})}}\left[1-\frac{1}{2p(k+M_{2})}\right].
\]
Combining \eqref{eq:Taylor} and (1) gives 
\[
(k+M_{2}+1)\left[1-\frac{1}{2p(k+M_{2})}\right]^{p}\geq(k+M_{2}+1)\left[1-\frac{p+1}{2p(k+M_{2})}\right]\geq k+M_{2}.
\]
A rearrangement of the above inequality gives 
\[
\alpha_{k+1}\geq\alpha_{k}(1-\gamma\alpha_{k}^{p})\geq\frac{1}{^{p}\sqrt{2p\gamma(k+M_{2})}}\left[1-\frac{1}{p(k+M_{2})}\right]\geq\frac{1}{^{p}\sqrt{2p\gamma(k+M_{2}+1)}},
\]
which is what we need.
\end{proof}

\subsection{\label{sub:2-hlfsp}The case of two halfspaces}

Let $\theta\in\mathbb{R}$ be such that $0<\theta<\pi/2$. Consider
the problem of projecting the point $x_{0}=(1,0)\in\mathbb{R}^{2}$
onto $H_{+}\cap H_{-}$, where $H_{+}$ and $H_{-}$ are halfspaces
in $\mathbb{R}^{2}$ defined by 
\[
H_{\pm}:=\{(u,v)\in\mathbb{R}^{2}:\pm v\geq u/\tan\theta\}.
\]
See Figure \ref{fig:eg-2-hlfsp}. It is clear that $P_{H_{+}\cap H_{-}}(x_{0})=(0,0)$.
We let $\mathbf{0}:=(0,0)$ to simplify notation. Haugazeau's algorithm
would be able to discover the two halfspaces in two steps and solve
the problem by quadratic programming. But suppose that somehow we
have an iterate $x_{1}$ that lies on the boundary of $H_{+}$ that
is close to $\mathbf{0}$. A similar situation arises in projecting
a point onto the intersection of many halfspaces for example. An analysis
of this modified problem gives us an indication of how Haugazeau's
algorithm can perform for larger problems.

\begin{figure}[!h]
\includegraphics[scale=0.25]{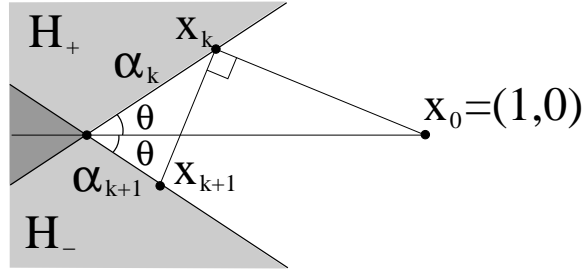}

\caption{\label{fig:eg-2-hlfsp}Illustration of example in Subsection \ref{sub:2-hlfsp}. }
\end{figure}

For our modified problem, the iterates $x_{i}$ would lie on the boundary
of either $H_{+}$ or $H_{-}$. For the iterate $x_{k}$, let $\alpha_{k}$
be the distance $\|x_{k}-(0,0)\|$. This is marked on Figure \ref{fig:eg-2-hlfsp}.
The cosine rule gives us the following equations. \begin{subequations}\label{eq-block:cos-rule}
\begin{eqnarray}
\|x_{k}-x_{k+1}\|^{2} & = & \alpha_{k}^{2}+\alpha_{k+1}^{2}-2\alpha_{k}\alpha_{k+1}\cos2\theta\label{eq:cos-rule1}\\
\|x_{0}-x_{k}\|^{2} & = & \alpha_{k}^{2}+1-2\alpha_{k}\cos\theta\label{eq:cos-rule2}\\
\|x_{0}-x_{k+1}\|^{2} & = & \alpha_{k+1}^{2}+1-2\alpha_{k+1}\cos\theta.\label{eq:cos-rule3}
\end{eqnarray}
\end{subequations}Pythagoras's theorem gives us $\|x_{k}-x_{k+1}\|^{2}+\|x_{0}-x_{k}\|^{2}=\|x_{0}-x_{k+1}\|^{2}$.
Together with the above equations, we have 
\begin{eqnarray*}
\alpha_{k}^{2}-2\alpha_{k}\alpha_{k+1}\cos2\theta-2\alpha_{k}\cos\theta & = & -2\alpha_{k+1}\cos\theta\\
\alpha_{k+1}[\cos\theta-\alpha_{k}\cos2\theta] & = & -\alpha_{k}^{2}+\alpha_{k}\cos\theta\\
\alpha_{k+1} & = & \alpha_{k}\frac{\cos\theta-\alpha_{k}}{\cos\theta-\alpha_{k}\cos2\theta}\\
 & = & \alpha_{k}\left(1-\alpha_{k}\frac{1-\cos2\theta}{\cos\theta-\alpha_{k}\cos2\theta}\right).
\end{eqnarray*}
Since $\{\alpha_{k}\}$ is a strictly decreasing positive sequence
which converges to zero, we have $\alpha_{k}\geq\alpha_{k}(1-\alpha_{k}\gamma)$
for all $k$ large enough, where $\gamma=\frac{1-\cos2\theta}{2\cos\theta}$.
By Lemma \ref{lem:lower-bdd-rate}, the convergence of $\{\|x_{k}-P_{H_{+}\cap H_{-}}(x_{0})\|\}$
to zero is at best $O(1/k)$. 

Let $f_{k}=\|x_{0}-x_{k}\|^{2}$. To see the rate of how $f_{k}$
converges to $1$, we note from \eqref{eq:cos-rule2} that $1-f_{k}=2\alpha_{k}\cos\theta-\alpha_{k}^{2}$.
Then the convergence rate of $f_{k}$ to $1$ is of $\Theta(1/k)$.

\subsection{The case of no linear metric inequality}

Let $p\geq1$ be some parameter. Consider the problem of projecting
the point $(1,0)\in\mathbb{R}^{2}$ onto the intersection of the sets
$C_{+}\cap C_{-}$, where 
\[
C_{\pm}=\{(u,v)\in\mathbb{R}^{2}:\pm v\geq|u|^{p}\}.
\]
The diagram for this problem is similar to that of the one in Subsection
\ref{sub:2-hlfsp}. The linear metric inequality is not satisfied
in this case. It is clear that the projection of $(1,0)$ onto $C_{+}\cap C_{-}$
is $(0,0)$. We try to show that the parameter $p$ can be made arbitrarily
large, so that the convergence of the iterates $x_{k}$ to $\mathbf{0}=(0,0)$
is arbitrarily slow. We let $x_{k}=(u_{k},v_{k})$. 
\begin{prop}
\label{prop:iters-outside-region}The iterates $x_{k}$ satisfy
\begin{equation}
x_{k}\notin\intr(C_{+})\cup\intr(C_{-}).\label{eq:intr-cond}
\end{equation}
\end{prop}
\begin{proof}
This is easily seen to be true for $k=1$. We now prove that \eqref{eq:intr-cond}
holds for all $k$ by induction. Without loss of generality, suppose
that for iterate $x_{k}$, its second coordinate $v_{k}$ is positive.
The next iterate $x_{k+1}$ is the intersection of the line passing
through $x_{k}$ perpendicular to $x_{0}-x_{k}$ and a supporting
hyperplane of $C_{-}$. It is therefore clear that $x_{k+1}\notin\intr(C_{-})$.
We also see that $u_{k+1}<u_{i}$. From the convexity of $C_{+}$,
if a point $x=(u,v)$ is such that $v>0$, $u<u_{1}$ and $x\notin\intr(C_{+})$,
then $\angle x_{0}x\mathbf{0}>\pi/2$. Given that $\angle x_{0}x_{k}\mathbf{0}>\pi/2$
and $x_{k}\notin\intr(C_{+})$, we have $x_{k+1}\notin\intr(C_{+})$
as well. 
\end{proof}
Next, we bound the rate of decrease of $u_{k}$.
\begin{prop}
\label{prop:lower-bdd-u-i}Continuing the discussion in this subsection,
we have $u_{k+1}\geq u_{k}\left(1-\frac{2u_{k}^{2p-1}}{1-u_{k}+u_{k}^{2p-1}}\right)$. \end{prop}
\begin{proof}
We assume without loss of generality that $x_{k}=(u_{k},v_{k})$ is
such that $v_{k}>0$. By Proposition \ref{prop:iters-outside-region},
we have $v_{k}\leq u_{k}^{p}$. 

Consider the point $\bar{x}_{k+1}$ defined by the intersection of
the line through $x_{k}$ perpendicular to $x_{k}-x_{0}$ and the
line passing through $\mathbf{0}$ and $(u_{k},-u_{k}^{p})$. One
can use geometrical arguments to see that $u_{k+1}\geq\bar{u}_{k+1}$,
where $u_{k+1}$ is the first coordinate of $x_{k+1}$ and $\bar{u}_{k+1}$
is the first coordinate of $\bar{x}_{k+1}$. We now bound $\bar{u}_{k+1}$
from below.

The point $\bar{x}_{k+1}$ is of the form $\lambda(u_{k},-u_{k}^{p})$.
From $[x_{k}-x_{0}]\perp[x_{k}-\bar{x}_{k+1}]$, we have 
\begin{eqnarray*}
\langle(u_{k}-1,u_{k}^{p}-0),(u_{k}-\lambda u_{k},u_{k}^{p}+\lambda u_{k}^{p})\rangle & = & 0\\
\lambda u_{k}(1-u_{k})+\lambda u_{k}^{2p}+(u_{k}-1)u_{k}+u_{k}^{2p} & = & 0\\
\lambda(1-u_{k}+u_{k}^{2p-1}) & = & 1-u_{k}-u_{k}^{2p-1}.
\end{eqnarray*}
This gives 
\[
u_{k+1}\geq\bar{u}_{k+1}=\lambda u_{k}=u_{k}\frac{1-u_{k}-u_{k}^{2p-1}}{1-u_{k}+u_{k}^{2p-1}}=u_{k}\left(1-\frac{2u_{k}^{2p-1}}{1-u_{k}+u_{k}^{2p-1}}\right),
\]
which ends our proof.
\end{proof}
We now make an estimate of how $\|x_{k}-x_{0}\|^{2}$ converges to
the optimal objective value of $1$ by analyzing $1-\|x_{k}-x_{0}\|^{2}$.
We have 
\begin{eqnarray*}
 &  & (1-u_{k})^{2}\leq\|x_{k}-x_{0}\|^{2}\leq(1-u_{k})^{2}+u_{k}^{2p}\\
 & \Rightarrow & 1-(1-u_{k})^{2}-u_{k}^{2p}\leq1-\|x_{k}-x_{0}\|^{2}\leq1-(1-u_{k})^{2}\\
 & \Rightarrow & 2u_{k}-u_{k}^{2}-u_{k}^{2p}\leq1-\|x_{k}-x_{0}\|^{2}\leq2u_{k}-u_{k}^{2}.
\end{eqnarray*}
 This means that $\{1-\|x_{k}-x_{0}\|^{2}\}$ converges to zero at
the same rate $\{2u_{k}\}$ converges to zero. By Lemma \ref{lem:lower-bdd-rate}
and Proposition \ref{prop:lower-bdd-u-i}, the convergence of $\{u_{k}\}$
to zero is seen to be at best $(\frac{1}{^{2p-1}\sqrt{k}})$. This
means that as we make $p$ arbitrarily large, the convergence of Haugazeau's
algorithm can be arbitrarily slow in the absence of the linear metric
inequality. It appears that enforcing the condition 
\[
C_{+}\cap C_{-}\subset\{x:\langle x_{0}-x_{k},x-x_{k}\rangle\leq0\}
\]
 makes Haugazeau's algorithm perform slower than the subgradient algorithm.

\bibliographystyle{amsalpha}
\bibliography{../refs}

\end{document}